\documentclass[12pt,final]{amsart}
\usepackage[utf8]{inputenc}

\usepackage[backend=bibtex,style=alphabetic,url=false, isbn=false, doi=false,maxbibnames=99]{biblatex}
\AtEveryBibitem{
    \clearfield{urlyear}
    \clearfield{urlmonth}
}
\renewbibmacro{in:}{}
\usepackage{amssymb,euscript,eufrak,amsmath,enumitem}
\DeclareMathAlphabet{\mathbbm}{U}{bbm}{m}{n}
\usepackage{tikz}
\usepackage{ytableau}
\usepackage[margin=3cm,footskip=25pt,headheight=20pt]{geometry}
\usepackage[colorlinks = true,
            linkcolor = blue,
            urlcolor  = blue,
            citecolor = blue,
            anchorcolor = blue]{hyperref}
\usepackage{listings}
\usepackage{fancyhdr,mathrsfs}
\usepackage[capitalize]{cleveref}
\crefalias{itheorem}{theorem}

  \usepackage{pdfsync}

\begin{document}
\newtheoremstyle{all}%  |?name>
  {11pt}%      |?Space above>
  {11pt}%      |?Space below>
  {\slshape}%         |?Body font>
  {}%         |?Indent amount>|ntnote1
  {\bfseries}% |?Theorem head font>
  {}%        |?Punctuation after theorem head>
  {.5em}%     |?Space after theorem head>|ntnote2
  {}%         |?Theorem head spec (can be left empty, meaning `normal')>

%\externaldocument[m-]{merged}[http://arxiv.org/pdf/1309.3796.pdf]
%\externaldocument[C-]{CB}[http://arxiv.org/pdf/1209.0051.pdf]

\theoremstyle{all}
\newtheorem{itheorem}{Theorem}
\newtheorem{theorem}{Theorem}[section]
\newtheorem{proposition}[theorem]{Proposition}
\newtheorem{corollary}[theorem]{Corollary}
\newtheorem{lemma}[theorem]{Lemma}
\newtheorem{assumption}[theorem]{Assumption}
\newtheorem{definition}[theorem]{Definition}
\newtheorem{ques}[theorem]{Question}
\newtheorem{conjecture}[theorem]{Conjecture}
\newtheorem{physics}[theorem]{Physics Motivation}

\theoremstyle{remark}
\newtheorem{remark}[theorem]{Remark}
\newtheorem{examplex}{Example}
\newenvironment{example}
  {\pushQED{\qed}\renewcommand{\qedsymbol}{$\clubsuit$}\examplex}
  {\popQED\endexamplex }
\renewcommand{\theexamplex}{{\arabic{section}.\roman{examplex}}}
\newcommand{\nc}{\newcommand}
\newcommand{\renc}{\renewcommand}
\numberwithin{equation}{section}
\renc{\theequation}{\arabic{section}.\arabic{equation}}

\newcounter{subeqn}
\renewcommand{\thesubeqn}{\theequation\alph{subeqn}}
\newcommand{\subeqn}{%
  \refstepcounter{subeqn}% Step subequation number
  \tag{\thesubeqn}% Label equation
}\makeatletter
\@addtoreset{subeqn}{equation}
\newcommand{\newseq}{%
  \refstepcounter{equation}% Step subequation number
}
\nc{\excise}[1]{}

\nc{\op}{\operatorname{op}}
\nc{\Sym}{\operatorname{Sym}}
\nc{\Symt}{S}
\nc{\tr}{\operatorname{tr}}
\newcommand{\Mirkovic}{Mirkovi\'c\xspace}
\nc{\tla}{\mathsf{t}_\la}
\nc{\llrr}{\langle\la,\rho\rangle}
\nc{\lllr}{\langle\la,\la\rangle}
\nc{\K}{\mathbbm{k}}
\nc{\Stosic}{Sto{\v{s}}i{\'c}\xspace}
\nc{\cd}{\mathcal{D}}
\nc{\cT}{\mathcal{T}}
\nc{\vd}{\mathbb{D}}
\nc{\Fp}{{\mathbb{F}_p}}
\nc{\lift}{\gamma}
\nc{\cox}{h}
\nc{\Aut}{\operatorname{Aut}}
\nc{\R}{\mathbb{R}}
\renc{\wr}{\operatorname{wr}}
  \nc{\Lam}[3]{\La^{#1}_{#2,#3}}
  \nc{\Lab}[2]{\La^{#1}_{#2}}
  \nc{\Lamvwy}{\Lam\Bv\Bw\By}
  \nc{\Labwv}{\Lab\Bw\Bv}
  \nc{\nak}[3]{\mathcal{N}(#1,#2,#3)}
  \nc{\hw}{highest weight\xspace}
  \nc{\al}{\alpha}
  \nc{\gK}{K}
  \nc{\gk}{\mathfrak{k}}
  
\newcommand{\LLoc}{\mathbb{L}\!\operatorname{Loc}}
\newcommand{\Rsecs}{\mathbb{R}\Gamma_\bS}

\newlength{\dhatheight}
\newcommand{\doublehat}[1]{%
    \settoheight{\dhatheight}{\ensuremath{\hat{#1}}}%
    \addtolength{\dhatheight}{-0.35ex}%
    \hat{\vphantom{\rule{1pt}{\dhatheight}}%
    \smash{\hat{#1}}}}

\newcommand{\dgmod}{\operatorname{-dg-mod}}
  \nc{\be}{\beta}
  \nc{\bM}{\mathbf{m}}
  \nc{\Bu}{\mathbf{u}}

  \nc{\bkh}{\backslash}
  \nc{\Bi}{\mathbf{i}}
  \nc{\Bm}{\mathbf{m}}
  \nc{\Bj}{\mathbf{j}}
 \nc{\Bk}{\mathbf{k}}
  \nc{\Bs}{\mathbf{s}}
\newcommand{\bS}{\mathbb{S}}
\newcommand{\bT}{\mathbb{T}}
\newcommand{\bt}{\mathbbm{t}}

\nc{\hatD}{\widehat{\Delta}}
\nc{\bd}{\mathbf{d}}
\nc{\D}{\mathcal{D}}
\nc{\mmod}{\operatorname{-mod}}  
\nc{\AS}{\operatorname{AS}}
\newcommand{\red}{\mathfrak{r}}

\nc{\RAA}{R^\A_A}
  \nc{\Bv}{\mathbf{v}}
  \nc{\Bw}{\mathbf{w}}
\newcommand{\arxiv}[1]{\href{http://arxiv.org/abs/#1}{\tt arXiv:\nolinkurl{#1}}}
  \nc{\Hom}{\mathrm{Hom}}
  \nc{\im}{\mathrm{im}\,}
  \nc{\La}{\Lambda}
  \nc{\la}{\lambda}
  \nc{\mult}{b^{\mu}_{\la_0}\!}
  \nc{\mc}[1]{\mathcal{#1}}
  \nc{\om}{\omega}
\nc{\gl}{\mathfrak{gl}}
  \nc{\cF}{\mathcal{F}}
%\renc{\thetheorem}{\arabic{theorem}}
 \nc{\cC}{\mathcal{C}}
  \nc{\Mor}{\mathsf{Mor}}
  \nc{\HOM}{\operatorname{HOM}}

  \nc{\sHom}{\mathscr{H}\text{\kern -3pt {\calligra\large om}}\,}
  \nc{\Ob}{\mathsf{Ob}}
  \nc{\Vect}{\operatorname{-Vect}}
  \nc{\slhat}[1]{\mathfrak{\widehat{sl}}_{#1}}
  \nc{\sllhat}{\slhat{\ell}}
  \nc{\slnhat}{\slhat{n}}
    \nc{\slehat}{\slhat{e}}
    \nc{\slphat}{\slhat{p}}
   
 \nc{\sle}{\mathfrak{sl}_e}
    \nc{\Th}{\theta}
  \nc{\vp}{\varphi}
  \nc{\wt}{\mathrm{wt}}
\nc{\te}{\tilde{e}}
\nc{\tf}{\tilde{f}}
\nc{\hwo}{\mathbb{V}}
\nc{\soc}{\operatorname{soc}}
\nc{\cosoc}{\operatorname{cosoc}}
 \nc{\Q}{\mathbb{Q}}
\nc{\LPC}{\mathsf{LPC}}
  \nc{\Z}{\mathbb{Z}}
  \nc{\Znn}{\Z_{\geq 0}}
  \nc{\ver}{\EuScript{V}}
  \nc{\Irr}{\operatorname{Irr}}
  \nc{\Res}[2]{\operatorname{Res}^{#1}_{#2}}
  \nc{\edge}{\EuScript{E}}
  \nc{\Spec}{\operatorname{Spec}}
  \nc{\tie}{\EuScript{T}}
  \nc{\ml}[1]{\mathbb{D}^{#1}}
  \nc{\fQ}{\mathfrak{Q}}
        \nc{\fg}{\mathfrak{g}}
        
        \nc{\fh}{\mathfrak{h}}
        
        \nc{\ft}{\mathfrak{t}}
        \nc{\fm}{\mathfrak{m}}
  \nc{\Uq}{U_q(\fg)}
        \nc{\bom}{\boldsymbol{\omega}}
\nc{\bla}{{\underline{\boldsymbol{\la}}}}
\nc{\bmu}{{\underline{\boldsymbol{\mu}}}}
\nc{\bal}{{\boldsymbol{\al}}}
\nc{\bet}{{\boldsymbol{\eta}}}
\nc{\rola}{X}
\nc{\wela}{Y}
\nc{\fM}{\mathfrak{M}}
\nc{\tfM}{\mathfrak{\tilde M}}
\nc{\fX}{\mathfrak{X}}
\nc{\fH}{\mathfrak{H}}
\nc{\fE}{\mathfrak{E}}
\nc{\fF}{\mathfrak{F}}
\nc{\fI}{\mathfrak{I}}
\nc{\qui}[2]{\fM_{#1}^{#2}}
\nc{\cL}{\mathcal{L}}
%\nc{\cO}{\mathcal{O}}
\nc{\ca}[2]{\fQ_{#1}^{#2}}
\nc{\cat}{\mathcal{V}}
\nc{\cata}{\mathfrak{V}}
\nc{\catf}{\mathscr{V}}
\nc{\hl}{\mathcal{X}}
\nc{\hld}{\EuScript{X}}
\nc{\hldbK}{\EuScript{X}^{\bla}_{\bar{\mathbb{K}}}}
\nc{\Iwahori}{\mathrm{Iwa}}
\nc{\hE}{\mathfrak{E}^{(1)}}
\nc{\Eh}{\mathfrak{E}^{(2)}}
\nc{\hF}{\mathfrak{F}^{(1)}}
\nc{\Fh}{\mathfrak{F}^{(2)}}
\
%\nc{\fE}{\mathfrak{E}}
%\nc{\fF}{\mathfrak{F}}

%\nc{\hl}{\mathfrak{X}}
\nc{\pil}{{\boldsymbol{\pi}}^L}
\nc{\pir}{{\boldsymbol{\pi}}^R}
\nc{\cO}{\mathcal{O}}
\nc{\Ko}{\text{\Denarius}}
\nc{\Ei}{\fE_i}
\nc{\Fi}{\fF_i}
\nc{\fil}{\mathcal{H}}
\nc{\brr}[2]{\beta^R_{#1,#2}}
\nc{\brl}[2]{\beta^L_{#1,#2}}
\nc{\so}[2]{\EuScript{Q}^{#1}_{#2}}
\nc{\EW}{\mathbf{W}}
\nc{\rma}[2]{\mathbf{R}_{#1,#2}}
\nc{\Dif}{\EuScript{D}}%{\operatorname{Dif}}
\nc{\MDif}{\EuScript{E}}
\renc{\mod}{\mathsf{mod}}
\nc{\modg}{\mathsf{mod}^g}
\nc{\fmod}{\mathsf{mod}^{fd}}
\nc{\id}{\operatorname{id}}
\nc{\compat}{\EuScript{K}}
\nc{\DR}{\mathbf{DR}}
\nc{\End}{\operatorname{End}}
\nc{\Fun}{\operatorname{Fun}}
\nc{\Ext}{\operatorname{Ext}}
\nc{\Coh}{\operatorname{Coh}}
\nc{\tw}{\tau}
\nc{\second}{\tau}
\nc{\A}{\EuScript{A}}
\nc{\Loc}{\mathsf{Loc}}
\nc{\eF}{\EuScript{F}}
\nc{\eE}{\EuScript{E}}
\nc{\LAA}{\Loc^{\A}_{A}}
\nc{\perv}{\mathsf{Perv}}
\nc{\gfq}[2]{B_{#1}^{#2}}
\nc{\qgf}[1]{A_{#1}}
\nc{\qgr}{\qgf\rho}
\nc{\tqgf}{\tilde A}
\nc{\Tr}{\operatorname{Tr}}
\nc{\Tor}{\operatorname{Tor}}
\nc{\cQ}{\mathcal{Q}}
\nc{\st}[1]{\Delta(#1)}
\nc{\cst}[1]{\nabla(#1)}
\nc{\ei}{\mathbf{e}_i}
\nc{\Be}{\mathbf{e}}
\nc{\Hck}{\mathfrak{H}}
%\nc{\fH}{\Hck}
%\nc{\fM}{\mathfrak{M}}
\renc{\P}{\mathbb{P}}
\nc{\bbB}{\mathbb{B}}
\nc{\ssy}{\mathsf{y}}
\nc{\cI}{\mathcal{I}}
\nc{\cG}{\mathcal{G}}
\nc{\cH}{\mathcal{H}}
\nc{\coe}{\mathfrak{K}}
\nc{\pr}{\operatorname{pr}}
\nc{\bra}{\mathfrak{B}}
\nc{\rcl}{\rho^\vee(\la)}
\nc{\tU}{\mathcal{U}}
\nc{\dU}{{\stackon[8pt]{\tU}{\cdot}}}
\nc{\dT}{{\stackon[8pt]{\cT}{\cdot}}}
\nc{\BFN}{\EuScript{R}}

\nc{\RHom}{\mathrm{RHom}}
\nc{\tcO}{\tilde{\cO}}
\nc{\Yon}{\mathscr{Y}}
\nc{\sI}{{\mathsf{I}}}
\nc{\sptc}{X_*(T)_1}
\nc{\spt}{\ft_1}
\nc{\Bpsi}{u}
\nc{\acham}{\eta}
\nc{\hyper}{\mathsf{H}}
\nc{\AF}{\EuScript{Fl}}
\nc{\VB}{\EuScript{X}}
\nc{\OHiggs}{\cO_{\operatorname{Higgs}}}
\nc{\OCoulomb}{\cO_{\operatorname{Coulomb}}}
\nc{\tOHiggs}{\tilde\cO_{\operatorname{Higgs}}}
\nc{\tOCoulomb}{\tilde\cO_{\operatorname{Coulomb}}}
\nc{\indx}{\mathcal{I}}
\nc{\redu}{K}
\nc{\Ba}{\mathbf{a}}
\nc{\Bb}{\mathbf{b}}
\nc{\Bc}{\mathbf{c}}
\nc{\Lotimes}{\overset{L}{\otimes}}
\nc{\AC}{C}
\nc{\rAC}{rC}\nc{\defr}{\operatorname{def}}

\nc{\rACp}{\mathsf{C}}
\nc{\ideal}{\mathscr{I}}
\nc{\ACs}{\mathscr{C}}
\nc{\Stein}{\mathscr{X}}
\nc{\pStein}{p\mathscr{X}}
\nc{\pSteinK}{\overline{\mathscr{X}}}
\nc{\No}{H}
\nc{\To}{Q}
\nc{\tNo}{\tilde{H}}
\nc{\tTo}{\tilde{Q}}
\nc{\flav}{\phi}
\nc{\tF}{\tilde{F}}
\newcommand{\cOg}{\mathcal{O}_{\!\operatorname{g}}}
\newcommand{\tcOg}{\mathcal{\tilde O}_{\!\operatorname{g}}}
\newcommand{\dOg}{D_{\cOg}}
\newcommand{\preO}{p\cOg}
\newcommand{\dpreO}{D_{p\cOg}}
\nc{\vertex}{\EuScript{V}(\Gamma)}
\nc{\Wei}{\EuScript{W}}
% Margin stuff from jason
%\oddsidemargin=0pt
%\evensidemargin=0pt
%\topmargin=0in
%\headheight=0pt
%\headsep=0pt
%\setlength{\textheight}{9in}
%\setlength{\textwidth}{6.5in}
\setcounter{tocdepth}{2}
\newcommand{\thetitle}{RoCK blocks for affine categorical representations}
\newcommand{\theshorttitle}{RoCK blocks for affine categorical representations}
\renc{\theitheorem}{\Alph{itheorem}}

\excise{
\newenvironment{block}
\newenvironment{frame}
\newenvironment{tikzpicture}
\newenvironment{equation*}
}

\baselineskip=1.1\baselineskip
\definecolor{bead}{gray}{0.2}

 \usetikzlibrary{decorations.pathreplacing,backgrounds,decorations.markings,shapes.geometric,decorations.pathmorphing,calc}
\tikzset{wei/.style={draw=red,double=red!40!white,double distance=1.5pt,thin}}
\tikzset{awei/.style={draw=blue,double=blue!40!white,double distance=1.5pt,thin}}
\tikzset{bdot/.style={fill,circle,color=blue,inner sep=3pt,outer
    sep=0}}
    \tikzset{
    weyl/.style={decorate, decoration={snake}, draw=black!50!green,
      very thick}}
\tikzset{fringe/.style={gray,postaction={decoration=border,decorate,draw,gray, segment length=4pt,thick}}}
\tikzset{old/.style={gray,thin}}
\tikzset{dir/.style={postaction={decorate,decoration={markings,
    mark=at position .8 with {\arrow[scale=1.3]{>}}}}}}
\tikzset{rdir/.style={postaction={decorate,decoration={markings,
    mark=at position .8 with {\arrow[scale=1.3]{<}}}}}}
\tikzset{edir/.style={postaction={decorate,decoration={markings,
    mark=at position .2 with {\arrow[scale=1.3]{<}}}}}}\begin{center}
\noindent {\large  \bf \thetitle}
\medskip

\noindent {\sc Ben Webster}\footnote{Supported by NSERC through a Discovery Grant.  This research was supported in part by Perimeter Institute for Theoretical Physics. Research at Perimeter Institute is supported in part by the Government of Canada through the Department of Innovation, Science and Economic Development Canada and by the Province of Ontario through the Ministry of Colleges and Universities.}\\  
Department of Pure Mathematics, University of Waterloo \& \\
 Perimeter Institute for Theoretical Physics\\
Waterloo, ON\\
Email: {\tt ben.webster@uwaterloo.ca}
\end{center}
\bigskip
{\small
\begin{quote}
\noindent {\em Abstract.}
Given a categorical action of a Lie algebra, a celebrated theorem of Chuang and Rouquier proves that the blocks corresponding to weight spaces in the same orbit of the Weyl group are derived equivalent, proving an even more celebrated conjecture of Brou\'e for the case of the symmetric group.   

In many cases, these derived equivalences are $t$-exact and thus induce equivalences of abelian categories between different blocks.  We call two such blocks ``Scopes equivalent.''  In this paper, we describe how Scopes equivalence classes for any affine categorification can be classified by the chambers of a finite hyperplane arrangement,  which can be found through simple Lie theoretic calculations.  We pay special attention to the largest equivalence classes, which we call RoCK, and show how this matches with recent work of Lyle on Rouquier blocks for Ariki-Koike algebras. We also provide Sage code that tests whether blocks are RoCK and finds RoCK blocks for Ariki-Koike algebras. 
\end{quote}
}

\section{Introduction}
A remarkable theorem of Chuang and Rouquier \cite{chuangDerivedEquivalences2008}, proving a conjecture of Brou\'e, shows that any two blocks of of modules over $\mathbb{F}_pS_n$ and $\mathbb{F}_pS_m$ with the same defect group are derived equivalent.  The proof of this theorem runs through a remarkable fact about the theory of symmetric groups: that it is best understood in terms of the representation theory of the affine Lie algebra $\slphat$, as pointed out in the title of \cite{grojnowskiAffineMathfraksl1999}.  In fact,  Chuang and Rouquier prove a much more general theorem, of which Brou\'e's conjecture is a special case.

They use the notion of a categorical representation of $\mathfrak{sl}_2$ (a strong $\mathfrak{sl}_2$-categorification, in their terminology) on a category $\cC$.  This consists of the data:
\begin{enumerate}
\item a decomposition $\cC\cong \bigoplus_{n\in \Z}\cC_n$;
\item functors $\eE\colon \cC_{n}\to \cC_{n+2}$ and $\eF\colon \cC_{n}\to \cC_{n-2}$;
\item certain special natural transformations between compositions of these functors that force $[\eE],[\eF]$ to satisfy the relations of $\mathfrak{sl}_2$ at the level of the Grothendieck group, with $\cC_n$ corresponding to the weight $n$ space.   
\end{enumerate}
Chuang and Rouquier then prove that for any categorical $\mathfrak{sl}_2$-representation, we have an equivalence of derived categories $D^b(\cC_n)\cong D^b(\cC_{-n})$ categorifying the action of the unique element of the Weyl group. Their proof of the Brou\'e conjecture for the symmetric group proceeds by successively applying different such equivalences for the $\mathfrak{sl}_2$-actions given by $i$-induction and $i$-restriction for different $i\in \mathbb{F}_p$.

These equivalences of derived categories are sometimes $t$-exact, and sometimes not.  If the Chuang-Rouquier equivalence is $t$-exact, then it induces an equivalence of categories $\cC_{n}\cong \cC_{-n}$.  We call these {\bf Scopes equivalences} since in the case of the symmetric group, they recover the Morita equivalences described by Scopes \cite{scopesCartanMatrices1991}.  Using these equivalences, Scopes showed that only finitely many different abelian categories up to equivalence appear amongst the blocks of a given defect.  In fact, ``most'' blocks are equivalent as abelian categories.  The blocks in this class are called {\bf RoCK} (for Rouquier-Chuang-Kessar) or simply {\bf Rouquier}. 

Our purpose in this paper is to describe how, in direct analogy with Chuang and Rouquier's approach to the Brou\'e conjecture, the theory of Scopes equivalences and RoCK blocks extend immediately to all categorical modules over an affine Lie algebra $\fg$.  For any categorical module $\cC$ over an affine Lie algebra $\fg$, we let its {\bf support} be the set of weights $\mu$ such that $\cC_{\mu}\neq 0$; similarly, for a $\fg$-module, we let its support be the set of weights with nonzero weight spaces.

Let $\bar{\fg}$ denote the corresponding finite-dimensional Lie algebra, $\bar{\fh}$ its Cartan, and $W,\bar{W}$ the corresponding affine and finite Weyl groups.  For each choice of $\cC$ and dominant weight $\la$ in the support of $\cC$ with stabilizer $W_{\la}$, we define a finite hyperplane arrangement in $\bar{\fh}$ given by certain translates of coroot hyperplanes; the arrangement only depends on the support of $\cC$ as a set (and thus will be the same for categorical actions with the same support).  We will call the chambers of this arrangement {\bf Scopes chambers}.   For each Weyl chamber of $\bar{W}$, there is a unique Scopes chamber which contains a translate of this Weyl chamber, which we call its {\bf RoCK} chamber.

We call two weights $\la$ and $\la'$ in the support of $\cC$ {\bf Scopes equivalent} if there is a $t$-exact Chuang-Rouquier equivalence between them.   Note that this is stronger than requiring that the blocks are equivalent as abelian categories; there are examples where the corresponding blocks are equivalent as categories but not via a Chuang-Rouquier equivalence.  In particular, a $t$-exact Chuang-Rouquier equivalence between blocks of Ariki-Koike algebras will send Specht modules to Specht modules and thus preserve decomposition matrices, while as noted in \cite[\S 1]{dellarcipreteEquivalenceDecomposition2023}, there are examples of Morita equivalences between such blocks which do not preserve Specht modules.

In analogy with Scopes' results, we show:
\begin{itheorem}\label{th:main}
  Consider any categorical module $\cC$ over an affine Lie algebra $\fg$ whose support is equal to that of a simple highest weight module
  and any dominant weight $\mu$ in its support:
  \begin{enumerate}
    \item The orbit $\{w\mu\}_{w\in W}$ of $\mu$ under the affine Weyl group is the union of finitely many Scopes equivalence classes.
    \item These classes are in bijection with orbits of the stabilizer $W_{\mu}$ of $\mu$ in $W$ on the Scopes chambers.  
   In particular,  the Scopes equivalence classes of a categorical module only depend on its support. 
    \end{enumerate}  
\end{itheorem}
   Under this bijection, certain equivalence classes (as many as $\# \bar{W}$) correspond to RoCK chambers. We call the weight spaces in these equivalence classes {\bf RoCK.} Part (1) of this theorem is proven in the case of Ariki-Koike algebras by Amara-Omari and Schaps \cite{amara-omariExternalVertices2021}; their proof uses the same ideas, but does not work out the combinatorics of Scopes chambers and leaves some important results implicit.  For example, our Lemma \ref{lem:Chuang-Rouquier1} is roughly equivalent to the proof of the Corollary to Proposition 3.2 in \cite{amara-omariExternalVertices2021}, but it is never stated there.  
   
   This definition is heavily inspired by work of Lyle on the case of Ariki-Koike algebras \cite{lyleRouquier2022} for the parameter $q$ a root of unity with quantum characteristic $e$, with the parameters $Q_j$ all lying in the set $q^{\Z}$.  Let $w_i$ be the multiplicity of $q^i$ in the multiset of $Q_j$'s.  The modules over the Ariki-Koike algebras $\oplus_{n}AK_{n}(\K,q,\Bw)\mmod$ categorify the simple module $V(\Lambda)=V(w_0\Lambda_0+\cdots+w_{e-1}\Lambda_{e-1})$ for $\slehat$.  In \S 3 of {\it loc.\ cit.}, Lyle defines a {\bf Rouquier block} for an Ariki-Koike algebra; this is a purely combinatorial property of a weight in the support of $V(\Lambda)$, so we can just as easily apply it to any other categorical module with this support.  Note that while ``Rouquier block'' and ``RoCK block'' are usually used synonymously, here we use them to distinguish Lyle's definition from ours.  
 
\begin{itheorem}\label{th:B}
For any categorical representation $\cC$ of $\fg=\slehat$ with support $V(\Lambda)$, the Scopes equivalence classes will coincide with those for the Ariki-Koike algebra, and a Scopes equivalence class is RoCK if and only if it contains a Rouquier weight.
\end{itheorem}
The proof of Theorem \ref{th:main} depends on the proof of a generalization of \cite[Conj. 1]{lyleRouquier2022} in the context of categorical representations (Lemma \ref{lem:Chuang-Rouquier}).   We should note that this result is well known to many experts; it could be regarded as a special case of \cite[Th. 6.6]{chuangDerivedEquivalences2008} or \cite[Prop. 8.4]{chuangPerverseEquivalences}, so our aim here is to make sure it is clearly stated for practitioners of modular representation theory who may find it useful, and to draw out its combinatorial consequences.   

In particular, the Scopes walls, which control Scopes equivalence and RoCKness, are very conducive to computer computation.  In particular, they are found using simple linear equalities on the number of boxes of a given residue and do not involve any enumeration of partitions.  We have written a Sage program (available \href{https://cocalc.com/share/public_paths/684c1deddca6be23d9d830298743118c07c493b6}{here on CoCalc})
that tests whether blocks are RoCK and constructs examples of RoCK blocks.  As discussed above, these computations apply not just to Ariki-Koike algebras but also to other categorifications with the same support, in particular, the cyclotomic $q$-Schur algebras and category $\mathcal{O}$ for the Cherednik algebras of $G(\ell,1,n)$.

\section{Background}
\subsection{Affine Lie algebras}
Let $\fg$ be an affine Lie algebra, $\fh$ its abstract Cartan, and $W$ its Weyl group. Since there are multiple variations of this algebra, we should clarify that we take the abstract Kac-Moody algebra defined by a given affine Cartan matrix.  That is, we assume that the simple coroots $\al_1^{\vee},\dots, \al_{n}^{\vee}\in \fh$ are linearly independent, as are the simple roots $\al_1,\dots, \al_{n}\in \fh^*$.  Since the Cartan matrix of $\fg$ has corank 1, this means that the span of the coroots has codimension 1 in $\fh$ and the span of the roots has codimension 1 in $\fh^*$.  

There is a unique primitive $\Z_{>0}$-linear combination $\delta=\sum \delta_i\al_i$ which is nonzero but perpendicular to all $\al_i^{\vee}$'s and similarly a unique primitive $\Z_{>0}$-linear combination $\delta^{\vee}=\sum \delta_i^{\vee}\al_i^{\vee}$ which is nonzero but perpendicular to all $\al_i$.

Let $\bar\fh_{\R}$ be the quotient of the $\R$-span of $\al_i^{\vee}$ by the $\R$-span of $\delta^{\vee}$.  The dual $\bar\fh_{\R}^*$ is the quotient of the $\R$-span of $\al_i$ by the $\R$-span of $\delta$.  The symmetrized Cartan matrix defines an inner product on $\bar\fh_{\R}$.  For each real root $\alpha$ of $\fg$, let $\bar{\alpha}$ be its image in $\bar\fh_{\R}^*$.  These images are integer multiples of roots in a system corresponding to a finite-dimensional simple Lie algebra $\bar\fg$.     Let $\bar{W}$ be the Weyl group of this finite root system.

Since it is central, the value of $\delta^{\vee}$ is constant on the weights of any irreducible representation of $\fg$.  This invariant of a representation is called its level.  The set of elements in $\fh^*_{\R}$ where $\delta^{\vee}>0$ is precisely the Tits cone of $\fg$, the union of all the weights in the $W$-orbit of the dominant Weyl chamber.  Readers may be familiar with the actions of $W$ on a fixed level coset for $c\in \R$:  
\[\fh^*_c=\{\la\in \fh^*/\mathbb C\delta \mid  \delta^{\vee}(\la)=c, \al_i^{\vee}(\la)\in \R\}.\]  This subspace is a coset of $\bar\fh_{\R}^*$, and thus inherits a metric. The reflection in any real root $\alpha$ of $\fg$ becomes the usual geometric reflection in the hyperplane 
\begin{equation}
     H_{\alpha}=\{\la\in \fh^*_c | \alpha^{\vee}(\la)=0\}.
\end{equation}
This hyperplane is a coset of the vanishing set of $\bar\alpha^{\vee}$ in $\fh^*_{\R}$.  

The hyperplanes $H_{\alpha}$ cut $\fh^*_c$ into chambers called {\bf alcoves}.  Each alcove corresponds to a Weyl chamber in the $W$-orbit of the dominant Weyl chamber.  In particular, we have the {\bf dominant alcove} 
\[A=\{\la\in \fh^*_c \mid \alpha^{\vee}(\la)>0\text{ for all positive roots  }\alpha\}.\]  The Weyl group $W$ acts simply transitively on the set of alcoves.  For any point $o\in \fh^*_c$, we can define a $\bar{W}$-action on $\fh^*_c$ by $\bar w(o+h)=o+\bar wh$ for $h\in \bar\fh_{\R}^*$.  Following convention, we take $o$ to be the unique point satisfying $\alpha_1^{\vee}(o)=\cdots =\al_{n-1}^{\vee}(o)=0$.  In particular, this point is a vertex of $A$.  The action of any element $w\in W$ on $\fh^*_c$ can be factored uniquely as product of a translation $\tau_h$ by an element of $\bar\fh_{\R}^*$, and an element $\bar{w}$ in $\bar{W}$.  Note that $\bar{w}\tau_h=\tau_{wh}\bar{w}$, so this factorization can be taken in either order.

The most important example for us will be $\slehat$, the affine Lie algebra obtained from $\sle((t))$ by taking the unique central extension and adding a loop element $\partial$ which acts by $t\frac{\partial}{\partial t}$. The torus $\fh$ can most easily be described as  \[\fh=\{(h_0,h_1\dots,h_{e+1})\in  \R^{e+2}\mid h_1+\cdots+h_e=0\}\] with the coroots
\[\al_e^\vee=(1,-1,0,\dots, 0,1,0)\qquad \al_i^\vee=(0,\dots, 1,-1, 0,\dots, 0). \]
The full set of positive roots is thus 
$\alpha_{ij;n}$ for $i,j\in [1,e]$, the vector with $h_i=1,h_j=-1$, and last coordinate $-n$, where $n\geq 0$ if $i<j$, and $n>0$ if $i>j$. Note that $\al_{e1;1}=\al_e$.  
We can identify the same space with $\mathfrak{h}^*$ via the inner product, and take the roots to be
\[\al_e=(0,-1,0,\dots, 0,1,-1)\qquad \al_i=(0,\dots, 1,-1, 0,\dots, 0). \]
Note that \[\delta=(0,0,\dots, 0,-1)\qquad \delta^{\vee}=(1,0,\dots, 0,0)\]
The reduced space is defined by \[\bar\fh\cong \bar \fh^* \cong \{(h_1,\dots, h_{e})\in \R^e\mid h_1+\cdots+h_e=0\}\]
where again, the metric is the usual one induced by the inner product on $\R^e$.    

The action of $s_i=s_{\alpha_i}$ on $\fh$ is thus given by the matrix $I-\al_i\cdot \al_i^{\vee}$, where we view $\al_i$ as a column vector and $\al_i^{\vee}$ as a row vector.  Thus, the reflections $s_i=s_{\al_i}$ for $i>0$ act by the usual permutation matrices on the coordinates $h_i$ and $h_{i+1}$, and $s_e=s_{\alpha_e}$ acts by
\[\begin{bmatrix} 1 & 0 & 0 &\cdots  &0&0 & 0\\
1 & 0 &0& \cdots&0 & 1 & 0\\
0 & 0 & 1 & \cdots & 0 & 0 & 0\\
\vdots & \vdots & \vdots &\ddots &\vdots &\vdots &\vdots \\
0& 0& 0 & \cdots & 1 & 0&0\\
-1 & 1 & 0 &\cdots & 0 & 0&0\\
1&  -1 & 0 & \cdots & 0& 1 & 1 \\
  \end{bmatrix}\]
  Considering the quotient by $\delta$, we obtain 
   \[\fh^*/\mathbb C\delta=\{(h_0,h_1\dots,h_{e})\in  \R^{e+1}\mid h_1+\cdots+h_e=0\},\] and the induced action on this quotient is obtained by 
  simply deleting the rightmost column and bottom row from each matrix.  Restricting to the level $\fh^*_c,$ we fix the coordinate $h_0=c$, so 
  \[s_i(c,h_1,\dots, h_e)=\begin{cases} (c, h_1,\dots, h_{i+1},h_{i},\dots,h_e) & i>0\\
  (c,h_{e}+c, h_2,\dots, h_{e-1}, h_1-c) & i=0\end{cases}\]
  Thus, $s_e$ acts by reflection in the line $h_1-h_e=c$.  More generally, the different roots of $\fg$ act by reflection in the lines $h_i-h_j=cm$ for $m\in \Z$.  
  
  From this perspective, the Weyl group $W$ acts by affine transformations; this can be seen from the fact that it is generated by the copy of $S_e$ generated by $s_1,\dots, s_{e-1}$, which acts linearly, and by translation by vectors in the root lattice of $\mathfrak{sl}_e$, i.e. the vectors \[X=\{(ca_1,\dots, ca_{e}) \mid a_i\in \Z, a_1+\cdots+a_e=0\}.\]
  The fundamental alcove is the set $A=\{h_1\geq h_2 \geq \cdots \geq h_e\geq h_1-c\}$; the last inequality follows by requiring $\alpha_e^{\vee}(c,h_1,\dots, h_e,*)=c+h_e-h_1\geq 0$.  
  
  For example, if $e=2$, then $\fh^*_c\cong \R$ is 1-dimensional, parameterized by $h_1$ with $h_2=-h_1$.  We have root hyperplanes at $h_1-h_2=cm$ for $m\in \Z$, or, equivalently, $h_1=c\frac{m}{2}$.  
  Thus, the alcoves are given by the intervals $A_m=[c\frac{m}{2},c\frac{m+1}{2}]$, and the fundamental alcove is defined by $h_2+c\geq h_1\geq h_2$ or equivalently $h_1\in A_0$.  The chambers $A_m$ with $m$ even are the images of $A_0$ under translations by $cm$ (these are the even elements of the Weyl group) and those with $m$ odd are the images of $A_0$ under a reflection at $h_1=c\frac{m+1}{4}$, the midpoint between these two chambers (these are all the odd elements of the affine Weyl group).
  
  If $e=3$, then the alcoves are equilateral triangles.  Around each point in the root lattice, there are 6 triangles around it; these form a hexagon.  For $0$, this hexagon is defined by $|h_i-h_j| \leq 1$ for all $i,j$.  The translations in the root lattice act freely transitively on the set of these hexagons.  These form the Voronoi tessellation of the root lattice.   
  
  This makes visible the factorization as the composition of a translation and an element of $\bar{W}$: A unique translation sends each alcove to one in the hexagon around the origin, which is the tip of one of the finite Weyl chambers, and then a unique element of $\bar{W}$ sends this to the fundamental alcove.  In Figure \ref{fig:alcove}, the fundamental alcove is colored green, its orbit under $\bar{W}$ is colored red, and all the elements of the root lattice that fit in the picture are marked with a black dot.  The other vertex points of alcoves are the elements of the weight lattice of $\mathfrak{sl}_3$ which don't lie in the root lattice; the other vertices of the fundamental alcove are $(\frac{1}{3},\frac{1}{3},-\frac{2}{3})$ and $(\frac{2}{3},-\frac{1}{3},-\frac{1}{3})$.
  \begin{figure}[h]
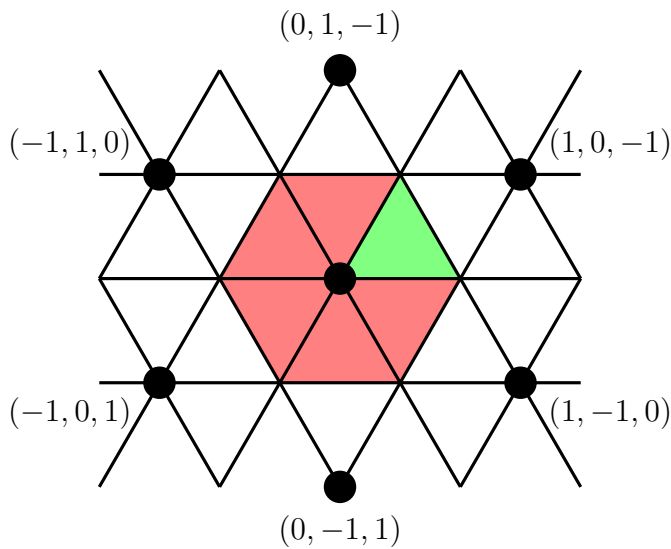

      \centering  \[\tikz[very thick,scale=1.6]{
  \fill[red!50!white] (1,0) -- (1/2,{sqrt(3)/2})--(-1/2,{sqrt(3)/2})-- (-1,0) -- (-1/2,{-sqrt(3)/2})--(1/2,{-sqrt(3)/2})-- cycle;
    \fill[green!50!white] (1,0) -- (1/2,{sqrt(3)/2})--(0,0)-- cycle;
  \draw (2,0)--(-2,0); \draw (2,{sqrt(3)/2})--(-2,{sqrt(3)/2});  \draw (2,{-sqrt(3)/2})--(-2,{-sqrt(3)/2});\draw (1,{sqrt(3)})-- (-1,-{sqrt(3)}); \draw (-1,{sqrt(3)})-- (1,-{sqrt(3)});
  \draw (2,{sqrt(3)})-- (0,-{sqrt(3)}); \draw (-2,{sqrt(3)})-- (0,-{sqrt(3)}); \draw (0,{sqrt(3)})-- (2,-{sqrt(3)}); \draw (0,{sqrt(3)})-- (-2,-{sqrt(3)});
    \draw (2,0)-- (1,-{sqrt(3)}); \draw (-2,0)-- (-1,-{sqrt(3)}); \draw (1,{sqrt(3)})-- (2,0); \draw (-1,{sqrt(3)})-- (-2,0); \node[draw=black,circle,fill,label=8:{$(1,0,-1)$}] at (3/2, {sqrt(3)/2}){};\node[draw=black,circle,fill,label=188:{$(-1,0,1)$}] at (-3/2, {-sqrt(3)/2}){};
    \node[draw=black,circle,fill,label=-8:{$(1,-1,0)$}] at (3/2, {-sqrt(3)/2}){};\node[draw=black,circle,fill,label=-188:{$(-1,1,0)$}] at (-3/2, {sqrt(3)/2}){};
      \node[draw=black,circle,fill,label=below:{$(0,-1,1)$}] at (0, {-sqrt(3)}){};\node[draw=black,circle,fill,label=above:{$(0,1,-1)$}] at (0, {sqrt(3)}){};
      \node[draw=black,circle,fill] at (0,0){};}\]
      \caption{Alcoves for $\widehat{\mathfrak{sl}}_3$.}
      \label{fig:alcove}
  \end{figure}

\subsection{Scopes chambers}
 Let $N$ denote any finite set of positive real roots.    %Let $\bar{N}=\{\bar{\alpha}\mid \alpha\in N \}$. 

The hyperplanes $H_{\alpha}$ for $\alpha\in N$ cut $\mathfrak{h}^*_c$ into finitely many chambers, which we call Scopes chambers.  Every Scopes chamber is defined by choosing $\epsilon_{\alpha}\in \{\pm 1\}$ for $\alpha\in N$, and considering the inequalities $C=\{h\mid \epsilon_{\alpha}\alpha^{\vee}(h)\geq 0\}$.  

\begin{lemma}\label{lem:same-scopes} \hfill
\begin{enumerate}
    \item 
We have $w\cdot\Delta_+\cap  N=w'\cdot\Delta_+\cap  N$ 
if and only if $wA$ and $w'A$ lie in the same Scopes chamber.  
\item If $wA$ and $w'A$ lie in the same Scopes chamber then there is a sequence $i_1,\dots, i_p$ such that for $w'=ws_{i_1}\cdots s_{i_p}$, and to $k=0,\dots,p $, $ws_{i_1}\cdots s_{i_k}A$ lies in the same Scopes chamber as $wA$ and $w'A$.
\end{enumerate}    
\end{lemma} 
\begin{proof}\hfill
\begin{enumerate}
    \item Note that $\alpha^{\vee}$ is positive on the alcove $wA$ if and only if $w^{-1}\alpha^{\vee}$ is positive on $A$, that is, if $w^{-1}\alpha\in \Delta$.  Thus, $ w\cdot\Delta_+\cap  N$ is exactly the subset of $N$ which is positive on $wA$.  By definition, this subset is the same on another alcove if and only if they are in the same Scopes equivalence class.
\item Let us prove this by induction on the number of root hyperplanes separating $wA$ and $w'A$.  When this number is 0, $w=w'$ and the claim is tautological.  Let $H_{\alpha}$ be a hyperplane that is a facet of $wA$ and separates it from $w'A$.  Then $H_{w^{-1}\alpha} $ is a facet of $A$, and so $w^{-1}\alpha$ is $\pm 1$ times a simple root $\alpha_{i_1}$.  This means that the reflection across $H_\alpha$ can be written as $s_{\alpha}=ws_{i_1}w^{-1}$.  Thus, $s_{\alpha}wA=ws_{i_1}A$ is a chamber separated from $w'A$ by only the hyperplanes that separate $wA$ from $w'A$, excluding $H_{\alpha}$.  Note that this means that it is not separated from $  wA$ or $w'A$ by the hyperplane defined by an element of $N$, and thus is in the same Scopes chamber.

By the inductive hypothesis applied to the chambers $ws_{i_1}A$ and $w'A$, we can find $i_2,\dots, i_p$ such that $w'=ws_{i_1}\cdots s_{i_p}$ and $ws_{i_1}\cdots s_{i_k}A$ are in the same Scopes chamber.  This completes the proof.\qedhere 
\end{enumerate}
\end{proof}

Every Scopes chamber $C$ has an asymptotic cone $\bar{C}=\{h\in \mathfrak{h} \mid th\in C\text{ for }t\gg 0\}$.  

\begin{lemma}
    The asymptotic cone of a Scopes chamber is always a face of the hyperplane arrangement $\bar{\alpha}=0$ for $\alpha\in N$.  For each finite Weyl chamber, there is a unique Scopes chamber that contains it in its asymptotic cone.
\end{lemma}
\begin{proof}
    The asymptotic cone of this chamber is defined by $\bar{C}=\{h\mid \epsilon_{\alpha}\bar{\alpha}(h)\geq 0\}$, since $\alpha(th)\geq 0$ for $t\gg 0$ iff $\bar \alpha(h)\geq 0$.  Therefore, this asymptotic cone will contain a Weyl chamber if and only if $\epsilon_{\alpha}=\epsilon_{\beta}$ whenever $\bar{\alpha}=\bar{\beta}$.  Given a Weyl chamber, we can define a corresponding $\epsilon_{\alpha}$ to have the same sign as $\bar{\alpha}$ on the Weyl chamber.  This will define the only Scopes chamber with our fixed Weyl chamber in its asymptotic cone.
\end{proof}

We call a Scopes chamber $C$ {\bf RoCK} if $\bar{C}$ contains a finite Weyl chamber, or equivalently, $C$ contains a translate of a finite Weyl chamber.  There are at most $\# \bar{W}$ RoCK chambers, one for each finite Weyl chamber; if the cone $\bar{C}$ contains a Weyl chamber $\mathfrak{c}$, we say that $C$ or an alcove in $C$ is {\bf RoCK for $\mathfrak{c}$.}

One helpful way to think about these different RoCK chambers is to factor your Weyl group element into a translation and an element of $\bar{W}$.  As discussed above, we can write $w=\tau_h\bar{w}$.  Since $\bar{w}$ sends $A$ to one of the other chambers adjacent to $o$, the element $h$ accounts for most of the hyperplanes that separate $wA$ from $A$.  More precisely, $w\cdot o=h+o$, so $wA$ must be one of the chambers that has $h+o$ as a vertex, and $\bar{w}$ only controls which of these chambers it is.  In particular, if $h+o$ is in the interior of a Scopes chamber, then all adjacent chambers are in the same Scopes chamber.  Thus, we can conclude that when $h$ is deep inside a Weyl chamber, that is, $\alpha^{\vee}(h) \gg 0 $ for all roots $\alpha$ positive on the chamber, then the alcove $wA$ will be in the corresponding RoCK chamber.  

%Another perspective we will use is that the RoCK Scopes chambers are distinguished by the property that for any $\al \in N$, we have that $\al^{\vee}$ has the same sign on $wA$ as the finite root $\bar{\al}$ on the corresponding Weyl chamber $\mathfrak{c}$.    That is:
%\begin{lemma}\label{lem:Weyl-chamber} The Scopes chamber containing 
%$wA$ is RoCK for a given Weyl chamber $\mathfrak{c}$ if and only if for all $\al\in N$, the root $w^{-1}\al$ is positive if and only if the finite root $\overline{\al}$ is positive on $\mathfrak{c}$.  
%\end{lemma} 

\subsection{Categorical actions}
While the fundamental input of this paper is a categorical action of a Lie algebra, we will not use this definition in a deep way, and thus will only give the basic facts about this notion that we need in this paper.  

A {\bf categorical action} of a Kac-Moody algebra $\fg$ is a representation of a particular 2-category $\tU(\fg)$.  This 2-category has:
\begin{enumerate}\setcounter{enumi}{-1}
    \item object set given by the weight lattice $\wela$ of $\fg$.  
    \item 1-morphisms generated by symbols $\eE_i,\eF_i$ for $i$ ranging over the simple roots of $\fg$.  These act on the weight lattice in the same way that Chevalley generators change weights:
    \[ \eE_i\colon \lambda\to \lambda+\al_i \qquad \eF_i\colon \lambda\to \lambda-\al_i\] 
    \item For our purposes, a detailed description of the 2-morphisms is not needed.  See \cite{Brundandef} for a detailed discussion of the different possible generating sets and relations.
\end{enumerate}
Thus, a categorical action is an assignment of a category $\cC_\la$ to each element of the weight lattice $\la$, a functor to each $\eE_i$ and $\eF_i$, and natural transformations to each 2-morphism. We will always consider the case where these categories $\cC_\la$ are abelian, and the functors $\eE_i,\eF_i$ are exact.   
We discuss the relevant examples of categorical actions in Section \ref{sec:examples} when we cover their RoCK blocks.  

In this paper, we will only consider actions that are {\bf integrable}.  That is, for every object $M$:
There is an integer $N$ such that $\eE_i^NM=0$ and $\eF_i^NM=0$.
 A hypothesis we assume in \cref{th:main,th:B} is that the category $\cC_{\la}$ is non-trivial only if the $\lambda$ weight space of a highest weight representation $V(\Lambda)$ is non-trivial; such a categorification is automatically integrable, since for any weight $\la$, the $\mathfrak{sl}_{2}$ root string for $\al_i$ through $\la$ is finite length, i.e. the weights $\la\pm N\al_i$ have trivial weight space for $N\gg 0$.

We will use two basic facts about integrable categorical actions, both of which follow from examination of the 2-morphisms:  \renewcommand{\theenumi}{\roman{enumi}}
\begin{enumerate}
    \item The power $\eF_i^k$ is the direct sum of $k!$ isomorphic summands, which we denote by $\eF_i^{(k)}$ and call the divided power functor.
    \item For each simple root, there is a chain complex $\Theta_i\colon \lambda\to s_i\lambda$ of 1-morphisms in $\tU$.  A version of this complex was defined in \cite[\S 6.1]{chuangDerivedEquivalences2008} but we use the definition in \cite[(3-4)]{Cauclasp} for more modern notation.  On an integrable representation, we can interpret $\Theta_i$ as a functor
$\Theta_i\colon D^b(\cC_{\la})\to D^b(\cC_{s_i\la})$, which is invertible up to homotopy and thus an equivalence of categories. 
\end{enumerate}

\section{RoCK blocks}
\subsection{Scopes equivalences from Chuang-Rouquier}
\label{sec:Scopes-CR}
Consider any integrable categorical module $\cC$ over $\fg$, and let $\nu$ be any weight of this module.  
Assume that $\al_i^{\vee}(\nu)=k>0$.  

\begin{lemma}\label{lem:Chuang-Rouquier1}
The functor $\eF_i^{(k)}\colon \cC_{\nu}\to \cC_{s_i\nu}$ is an equivalence of abelian categories (that is, a Scopes equivalence) if and only if the category $\cC_{\nu+\al_i}$ is trivial.
\end{lemma}

This lemma is a restatement of one special case of \cite[Prop. 8.4]{chuangPerverseEquivalences}: the Chuang-Rouquier equivalences are perverse and will be $t$-exact if and only if the perversity function is 0 for all simples. This is equivalent to the condition that $\cC_{\nu+\al_i}$ is trivial.  We will give a more direct proof here that does not depend on the notion of a perverse equivalence. 
\begin{proof}
    $\Rightarrow$: If $\eF_i^{(k)}$ is an equivalence of categories, then $\eE_i^{(k)}$ is its left and right adjoint, and thus the inverse equivalence.  Thus, for any object $M\in \cC_{\nu}$, we have $M\cong  \eE_i^{(k)} \eF_i^{(k)}M$.  On the other hand, if $\cC_{\nu+\al_i}\neq 0$, then there is a highest weight object $N$ in $\cC_{\nu +r\al_i}$ for some $r>0$.  The object $\eF_i^{(r)}N$ is nonzero, and 
\[\eE_i^{(k)}\eF_i^{(k)}\eF_i^{(r)}N\cong \eE_i^{(k)}\eF_i^{(k+r)}N^{\oplus \binom{k+r}{k}}\cong \eF_i^{(r)}N^{\oplus\binom{k+r}{k}^2}\]
applying \cite[Lem. 4.1]{Cauclasp} with $\mu=k+2r,b=k+r,a=k$.  This contradicts the claim that $\eF_i^{(k)}$ is an equivalence.

$\Leftarrow$:  Since $\cC_{\nu+\al_i}$ is trivial, every object in $\cC_{\nu}$ is killed by $\eE_i$, that is, it is highest weight.  This implies that the only term of the Rickard complex $\Theta_i$ which acts nontrivially is $\eF_i^{(k)}$ in homological degree 0.  Thus, on $\cC_{\nu}$, the actions of $\Theta_i$ and the derived functor of $\eF_i^{(k)}$ coincide.  The former is an equivalence of derived categories, and the latter is exact in the usual $t$-structure.  This shows that $\eF_i^{(k)}$ is an equivalence of abelian categories as desired.
\end{proof}

Now, we restrict attention to the case where $\nu$ is positive level (i.e. $\delta^{\vee}(\nu)>0$). In this case, there is a unique dominant weight $\mu=w\nu$ in the orbit of $\nu$.    Consider the set $N_{\mu}=\{\al\in \Delta^{+}\mid \cC_{\mu+\alpha}\neq 0\}$.

The failure of the functor $\Theta_i$ to preserve the $t$-structure depends in a precise way on the structure of the representation of the $\mathfrak{sl}_2$ generated by $E_i,F_i$ on the root string $\nu +m\al_i$ for all $m\in \Z$.  Using the action of $W$, this is the same as the structure of the root string through $\mu$ with respect to the root $w\al_i$.  In particular:
\begin{lemma}\label{lem:Chuang-Rouquier}
  The following are equivalent:
  \begin{enumerate}
    \item The functor $\eF_i^{(k)}\colon \cC_{\nu}\to \cC_{s_i\nu}$ is an equivalence of abelian categories (that is, a Scopes equivalence).
    \item The category $\cC_{\mu+w\al_i}$ is trivial, i.e. $w\al_i\notin N_{\mu}$.
    \item The alcoves $wA$ and $ws_iA$ lie in the same Scopes chamber.
  \end{enumerate}
\end{lemma} 
\begin{proof}
$(i)\Leftrightarrow(ii)$: By \cite{chuangDerivedEquivalences2008}, since $w(\nu+\al_i)=\mu+w\al_i$, the categories $\cC_{\mu+w\al_i}$ and $\cC_{\nu+\al_i}$ are derived equivalent. Thus, if one is trivial, the other is.  The result now follows from Lemma \ref{lem:Chuang-Rouquier1}.  

$(ii)\Leftrightarrow(iii)$: The only hyperplane separating $wA$ and $ws_iA$ is $H_{w\al_i}$, so these lie in a common Scopes chamber if and only if $w\al_i\notin N_{\mu}$.  
  \end{proof}
This equivalence will induce a bijection of simple modules, matching the action of the Kashiwara operator $f_i^k$ in the crystal structure on simples.  
  
Combining Lemma \ref{lem:same-scopes} and Lemma \ref{lem:Chuang-Rouquier}, we arrive at the main result of this paper.  Consider $w,w'\in W$ and as above, $\mu$ dominant and $\nu=w^{-1}\mu,\nu'=(w')^{-1}\mu$.
\begin{theorem}
 If $wA$ and $w'A$ lie in the same Scopes chamber for the set $N_{\mu}$, then we have an equivalence of abelian categories $\cC_{\nu'}\cong \cC_{\nu}$.  
\end{theorem}

Furthermore, some important examples of interest to us, such as Schur algebras in positive characteristic, (cyclotomic) $q$-Schur algebras, and categories $\cO$ for Cherednik algebras of $G(\ell,1,n)$ are 
 {\it highest weight categorifications} in the sense of \cite{losevHighestWeight2013}.  
In particular, the category $\mathcal{C}$ has a highest weight structure with standard objects $\Delta_L$ for $L\in I=\Irr(\mathcal{C})$.  Furthermore, the indexing set $I$ has a partition $I=\cup_{a\in \mathcal{A}^{(i)}}I_a^{(i)}$ indexed by a smaller set $\mathcal{A}^{(i)}$ such that there is a bijection $I_a^{(i)}\cong\{-,+\}^{n_a}$ for some $n_a\geq 0$.  The functor $\eE_i^{(k)}$ (resp. $\eF_i^{(k)}$) sends $\Delta_L$ for $L\in I_a^{(i)}$ to module filtered by $\Delta_{L'}$ for $L'\in I_a^{(i)}$, obtained all different ways of turning a subset of size $k$ of the $+$'s to $-$'s in the sign sequence of $L$ (resp. $+$'s to $-$'s).  

For Schur algebras in characteristic $p$, the sets $I_a^{(i)}$ are the subsets of partitions related by only adding and removing boxes of residue $i$, and similarly for (cyclotomic) $q$-Schur algebras, and categories $\cO$ for Cherednik algebras of $G(\ell,1,n)$.

 Let $\mathcal{B}\subset \mathcal{C}$ be a Serre subcategory of $\mathcal{C}$ closed under the action of $\eE_i^{(k)}$ and $\eF_i^{(k)}$.  If $\mathcal{C}=A\mmod$ for a finite dimensional algebra, we can write $\mathcal{B}=\{M\in A\mmod \mid eM=0\}$ for some idempotent $e$.  The quotient $\bar{\mathcal{C}}=\mathcal{C}/\mathcal{B}=eAe\mmod$ is no longer necessarily a highest weight category, but it still has distinguished images of the standard modules, which we call  {\bf Specht modules} $S_L$ for $L\in I$.  By definition, the functors $\eE_i^{(k)}$ and $\eF_i^{(k)}$ send Specht-filtered modules to Specht-filtered modules with the same combinatorics as standards in $\mathcal{C}$.

  This includes examples such as the modules over the symmetric group, Hecke, and Ariki-Koike algebras are quotient categories (that is, modules over a corner algebra $eAe$) of a highest weight categorification.    
The multiplicities of simple modules in Specht modules in these cases, called {\bf decomposition numbers}, play a central role in modular representation theory.  We can compare these in different weight spaces in the same orbit of the Weyl group using the observation that if $\al_i^{\vee}(\nu)=k>0$ the Kashiwara operator $\tilde{f}_i^k$ induces a bijection $\Irr(\mathcal{C}_\nu)\cong \Irr(\mathcal{C}_{s_i\nu})$ even if the Chuang-Rouquier equivalence is not exact, and so any reduced word $w=s_{i_m}\cdots s_{i_1}$ induces a bijection $ \Irr(\mathcal{C}_\mu)\cong \Irr(\mathcal{C}_{w\mu})$.

\begin{proposition}
	Let $\mathcal{C}$ be a highest weight categorification with subcategory $\mathcal{B}$ as above, and assume that $\nu,i$ are chosen so that the conditions of Lemma \ref{lem:Chuang-Rouquier1} hold.  \begin{enumerate}
		\item  The functor $\mathcal{F}_i^{(k)}\colon \mathcal{C}_{\nu}\to \mathcal{C}_{s_i\nu}$ is an equivalence of highest weight categories sending $\mathcal{F}_i^{(k)}(\Delta_{L})=\Delta_{\tilde{f}_i^kL}$.  Thus, we also have $\mathcal{F}_i^{(k)}(S_{L})=S_{\tilde{f}_i^kL}$. 
		\item We have an equality of  decomposition numbers $[S_L:M]=[S_{\tilde{f}_i^kL}:\tilde{f}_i^kM]$ for all $L,M\in \Irr(\mathcal{C})$.
	\end{enumerate} 
\end{proposition}
\begin{proof}
	\begin{enumerate}[wide]
		\item If the conditions of Lemma \ref{lem:Chuang-Rouquier} hold, then any simple module $L$ of weight $\mu$ must correspond to the sign sequence $\{-^k\}$, since if there were a plus sign, then $\eE_iL\neq 0$, which Lemma \ref{lem:Chuang-Rouquier1} shows implies that we do not obtain a Scopes equivalence.  The highest weight conditions imply that $\eF_i^{(k)}\Delta_L=\Delta_{\tilde{f}_i^kL}$. This implies that $ \eF_i^{(k)}$ is an equivalence of abelian categories and sends standards to standards and thus is an equivalence of highest weight categories.  The connection between Specht modules follows immediately.
		\item This follows immediately from the functor sending simples to simples and Spechts to Spechts.  \qedhere
	\end{enumerate}
\end{proof}
%If, as before, we assume that $\al_i^{\vee}(\mu)=k>0$, and one of the conditions of Lemma \ref{lem:Chuang-Rouquier1} holds, then for all multipartitions $\xi$ of weight (i.e. residue) $\nu$,  then $\mathcal{F}_i^{(k)}S_{\xi}=S_{\tilde{f}_i^k\xi}$ for the Kashiwara operator $\tilde{f}_i$.  Applying this implies that:
% \begin{corollary}
%  Any $t$-exact Chuang-Rouquier equivalence between weight spaces $\mu,w\mu$ on a quotient of a highest weight categorification sends the Specht modules to Specht modules, via a bijection .   Thus Scopes equivalent blocks will have the same decomposition numbers, where we match both simples and Spechts using Kashiwara operators.  
% \end{corollary}
Applying this inductively, if two weight spaces are Scopes equivalent, then their decomposition matrices are the same.   For Ariki-Koike algebras, this is proven with a roughly equivalent proof in \cite[Prop. 5.5]{dellarcipreteEquivalenceDecomposition2023}.
   
In all the cases discussed above, we can index simples by abacus diagrams, and Lemma \ref{lem:Chuang-Rouquier} will apply when for every object in our block, there is no way to push a bead from the $(i+1)$st runner to the $i$th, i.e., for every bead on the $(i+1)$st, the position to its left on the $i$th runner is occupied.  Note that a necessary condition for this to hold in the Ariki-Koike case is given in 
 \cite[Lem. 5.1]{dellarcipreteEquivalenceDecomposition2023}.

In this case, $k$ is the number of beads on the $i$th runner where the spot to the right is empty, and the Kashiwara operator $f_i^k$ acts by pushing these beads right, that is, by swapping the runners.  As usual, when $i=e$, we have to interpret all these statements as comparing the $e$th and first runners with a shift.  

For example, in the picture below, we first show the $i$ and $(i+1)$ runners of the diagram where the Kashiwara operator ${e}_i$ acts non-trivially, to give the second picture where it acts non-trivially.  So, Lemma \ref{lem:Chuang-Rouquier} applies when all abacus diagrams in your block look like the second diagram and never like the first.  As discussed above, the Kashiwara operator $f_i^2$ sends the second diagram to the third by moving the two dots that have space to move right.  This has the same effect as swapping the runners.\smallskip

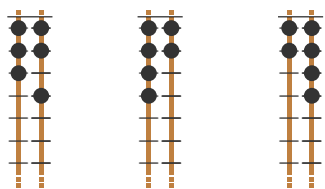
\centerline{
\begin{tikzpicture}\tikzset{yscale=0.3,xscale=0.3}
\foreach \k in {0,1} {
\fill [color=brown] (\k-.1,0)--(\k-.1,7) -- (\k+0.1,7)-- (\k+.1,0)--(\k-.1,0);
\fill[color=brown](\k-.1,7.1)--(\k+.1,7.1)--(\k+.1,7.3)--(\k-.1,7.3)--(\k-.1,7.1);
\fill[color=brown](\k-.1,-0.1)--(\k+.1,-0.1)--(\k+.1,-0.3)--(\k-.1,-0.3)--(\k-.1,-0.3);
\fill[color=brown](\k-.1,-0.4)--(\k+.1,-0.4)--(\k+.1,-0.6)--(\k-.1,-0.6)--(\k-.1,-0.6);
}
\foreach \k in {4.5,5.5,6.5} {
\fill[color=bead] (0,\k) circle (10pt);}
\foreach \k in {3.5,5.5,6.5} {
\fill[color=bead] (1,\k) circle (10pt);}
\draw(-.5,7)--(1.5,7); 
\foreach \a in {0,1,} {
\foreach \k in {.5,...,6.5}{
\node at (\a,\k)[color=bead]{$-$};};};
\end{tikzpicture}\qquad 
\begin{tikzpicture}\tikzset{yscale=0.3,xscale=0.3}
\foreach \k in {0,1} {
\fill [color=brown] (\k-.1,0)--(\k-.1,7) -- (\k+0.1,7)-- (\k+.1,0)--(\k-.1,0);
\fill[color=brown](\k-.1,7.1)--(\k+.1,7.1)--(\k+.1,7.3)--(\k-.1,7.3)--(\k-.1,7.1);
\fill[color=brown](\k-.1,-0.1)--(\k+.1,-0.1)--(\k+.1,-0.3)--(\k-.1,-0.3)--(\k-.1,-0.3);
\fill[color=brown](\k-.1,-0.4)--(\k+.1,-0.4)--(\k+.1,-0.6)--(\k-.1,-0.6)--(\k-.1,-0.6);
}
\foreach \k in {3.5,4.5,5.5,6.5} {
\fill[color=bead] (0,\k) circle (10pt);}
\foreach \k in {5.5,6.5} {
\fill[color=bead] (1,\k) circle (10pt);}
\draw(-.5,7)--(1.5,7); 
\foreach \a in {0,1,} {
\foreach \k in {.5,...,6.5}{
\node at (\a,\k)[color=bead]{$-$};};};
\end{tikzpicture}
\qquad 
\begin{tikzpicture}\tikzset{yscale=0.3,xscale=0.3}
\foreach \k in {0,1} {
\fill [color=brown] (\k-.1,0)--(\k-.1,7) -- (\k+0.1,7)-- (\k+.1,0)--(\k-.1,0);
\fill[color=brown](\k-.1,7.1)--(\k+.1,7.1)--(\k+.1,7.3)--(\k-.1,7.3)--(\k-.1,7.1);
\fill[color=brown](\k-.1,-0.1)--(\k+.1,-0.1)--(\k+.1,-0.3)--(\k-.1,-0.3)--(\k-.1,-0.3);
\fill[color=brown](\k-.1,-0.4)--(\k+.1,-0.4)--(\k+.1,-0.6)--(\k-.1,-0.6)--(\k-.1,-0.6);
}
\foreach \k in {5.5,6.5} {
\fill[color=bead] (0,\k) circle (10pt);}
\foreach \k in {3.5,4.5,5.5,6.5} {
\fill[color=bead] (1,\k) circle (10pt);}
\draw(-.5,7)--(1.5,7); 
\foreach \a in {0,1,} {
\foreach \k in {.5,...,6.5}{
\node at (\a,\k)[color=bead]{$-$};};};
\end{tikzpicture}}

\begin{definition}\label{def:our-RoCK}
    We call a weight space category {\bf RoCK} or a {\bf RoCK block} if the corresponding Scopes chamber is RoCK.  By construction, all RoCK blocks for a given Weyl chamber are equivalent as abelian categories.  
\end{definition}

We do not aim here to comprehensively address the question of when weight categories are equivalent; nothing we have written above precludes the existence of an equivalence $\cC_{\nu'}\cong \cC_{\nu}$ if $wA$ and $w'A$ do not lie in the same Scopes chamber.  There are two obvious situations where this can happen: 
\begin{enumerate}
  \item If $\cC_{\mu+w\al_i}$ is nontrivial, but the representation of the root $\mathfrak{sl}_2$ for $w\al_i$ generated by the $\mu$ weight space is isotypic (all its highest weight vectors are of the same weight), then some shift of the Chuang-Rouquier functor $\Theta_i$ is exact, and thus induces an equivalence of abelian categories.  The easiest way this can happen is condition (2) of Lemma \ref{lem:Chuang-Rouquier}, but it could be that all these highest weight vectors have weight $\mu+rw\alpha$ for some $r$.  In this case, every projective object $P$ in $\cC_{\nu}$ can be written as $P=\eF_i^{(r)}P'$, and there is an equivalence of abelian categories sending $ P\mapsto \eF_i^{(k+r)}P'$.   
  
We could tighten our results a bit by defining $N'_{\mu}$ to be the set of roots which do not have this isotypic property and only considering Scopes chambers with respect to this set. This suffers from the issue of not being only determined by the support of the representation, and in the vast majority of cases, $N_{\mu}$ and $N'_{\mu}$ will be equivalent.   
  \item 
If $\mu$ has nontrivial stabilizer, then we can have $\nu'=\nu$, while $wA$ and $w'A$ do not lie in the same Scopes chamber.  This tells us that the induced autoequivalence of $\cC_{\nu'}$ will not be exact, but this does not change that the categories are exactly equal.  

Of course, the stabilizer of $\mu$ is generated by $s_{\alpha}$ for $\alpha^{\vee}(\mu)=0$, which can only happen if $\alpha$ is simple.    
Let $W_\mu$ be the stabilizer; we can simplify the determination of the different blocks that appear in this case by only considering $wA$ in a fundamental domain of $W_\mu$; the most natural choice is to consider the chamber cut out by $\alpha^{\vee}=0$ for the $\alpha$ such that $\alpha^{\vee}(\mu)=0$ which contains the fundamental alcove.  This has the effect of requiring $w$ to be a shortest right coset representative for $W_\mu$.  \end{enumerate}

\subsection{Irreducible support}
\label{sec:irreducible}

Recall that the {\bf support} of a categorical module $\cC$ for $\fg$ is the set of weights $\mu$ such that $\cC_\mu \neq 0$; by analogy, for a usual linear representation $V$ of $\fg$, we will call the set of weights with nonzero weight spaces the support of $V$.  Most of the categorical modules appearing in representation theory (of which the author knows) satisfy the following property: 
\begin{enumerate}
    \item [$(*)$] The support of $\cC$ is equal to the support of a simple highest weight module $V(\Lambda)$.  
\end{enumerate}
This will happen in many cases where the representation $K^0_{\mathbb{C}}(\cC)$ is very much not irreducible, but the support of one simple summand contains the support of all other summands.  For example, if we consider a tensor product $V_{\Lambda}\otimes V_{\Lambda'}$, this is typically not irreducible, but its support is the same as the Cartan component $V_{\Lambda+\Lambda'}$.

This is a particularly nice situation, since we can encode the support of $\cC$ in purely combinatorial terms.  
  Let $\leq$ denote the usual root order on weights, that is, the transitive closure of $\mu-\al_i\leq \mu$ for all $\mu, i$.  Let us recall one of the standard characterizations of the support of a simple highest weight module:
 \begin{lemma}[\mbox{\cite[Prop.\ 11.3(a)]{kacInfiniteDimensional1990}}]\label{lem:simple-support}
     The following are equivalent for a weight $\mu$:
     \begin{enumerate}
         \item The weight $\mu$ is in the support of $V(\Lambda)$.
         \item There is an element $w\in W$ such that $w\mu$ is dominant and $w\mu\leq \Lambda$.
         \item For all $w\in W$, we have $w\mu\leq \Lambda$.
     \end{enumerate}
 \end{lemma}
Therefore, each $\mu$ in the support of $V(\Lambda)$ is of the form
\[\mu=\Lambda-\sum_{i=1}^{e} b_i\al_i\] and any dominant $\mu$ with this form is in the support by $(2)\Rightarrow(1)$.  Thus, given a simple root $\alpha$, we wish to determine if $\mu+\alpha$ is in the support of $\cC$, i.e. if $\alpha\in N_{\mu}$.  This can only happen if $\mu+\al\leq \Lambda$, and in some corner cases, $\mu+\al$ will not be dominant, and we need to check the dominant weight in its orbit also satisfies $w(\mu+\al)\leq \Lambda$.  Note that this implies that $N_{\mu}$ is finite in this case.  

This is easy to check for one orbit, but the reader will correctly note that the set of positive roots is infinite, making it hard to check by hand whether this holds for each of them.  However, we can exploit the fact that addition by $\delta$ commutes with every element of $W$.  
Fix a set $\{\beta_1,\dots, \beta_r\}\subset \Delta^+$ of positive roots with the property that every affine root is of the form $\pm \beta_i+k\delta$ for a unique choice of sign, $i\in [1,r]$ and $k\in \Z$.    Note that in twisted cases, not every element of this form is a root. So, we have that $w\cdot (\mu\pm \beta_i+k\delta)$ is dominant if and only if $w\cdot (\mu\pm \beta_i)$ is.  For each $i\in [1,r]$, we therefore have integers $k_i^{\pm}$ defined as the largest integers such that
\[w(\mu\pm \beta_i)+k_i^{\pm}\delta\leq \Lambda \qquad \forall w\in W.\]

Of course, checking this for all $w$ is equivalent to checking it only when $w(\mu\pm \beta_i)$ is dominant by Lemma \ref{lem:simple-support}.  This implies that a positive root $\pm \beta_i+k\delta$ lies in $N_\mu$ if and only if $k\leq k_i^{\pm}$.  However, as we noted earlier, there might be values of $k$ where this is not a positive root.

  If we assume our affine Lie algebra is simply laced, that is, it is of type $\widehat{A},\widehat{D}$ or $\widehat{E}$, then we can simplify by choosing our $\beta_i$ to the positive roots in the finite type subalgebra, so the set of positive roots can be described as
  \[\Delta^+=\bigcup_{\substack{i\in [1,r]\\ m\in \Z_{\geq 0}}}\{\beta_i+m\delta,-\beta_i+(m+1)\delta\} \]
and the resulting description of the Scopes chambers has a particularly nice form:
\begin{lemma}\label{lem:Scopes-walls}
  We have an equality $N_{\mu}=\{\pm \beta_i+m\delta\mid m\leq k_i^{\pm}\}$.  The walls of the Scopes chambers are defined by 
  \[\beta_i^{\vee}(\lambda)=m\delta^{\vee}(\lambda) \text{ for all } m\in [-k_i^+,k_i^-].\]
\end{lemma}

Combining Lemmata \ref{lem:Chuang-Rouquier} and \ref{lem:Scopes-walls}, we arrive at an algorithm to test whether a block is RoCK and to construct blocks that are RoCK for each Weyl chamber.  Assume that the property $(*)$ holds and fix a weight $\mu$; the steps of our algorithm are:
\begin{enumerate}[wide]
  \item {\bf Find the dominant element of the orbit:} Let $w\in W$ be the element of minimal length such that $\mu'=w\mu$ is dominant.  Also, find the alcove $A'=wA$. Algorithmically, we can do this by checking the sign of $\al_i^{\vee}(\mu)$ for all $i$.  If all of these are $\geq 0$, then $\mu$ is dominant, and we are done.  Otherwise,  replace $\mu$ by $s_i\mu$ for any $i$ with $\al_i^{\vee}(\mu)<0$.  Note that this terminates since $\mu<s_i\mu\leq \Lambda$.  
  \item {\bf Find the set $N_{\mu'}$:} For each $\beta_i$, we consider $\mu\pm\beta_i$, find the dominant element of its $W$-orbit, and use this to find the bounds $k_i^{\pm}$, which specify all the Scopes walls.
  \item {\bf Check the sign on $A'$}: Check the value of the ratio $\rho=\frac{\beta_i^{\vee}}{\delta^{\vee}}$ on any element of $A'$.  If $-k_i^+\leq \rho \leq k_i^-$ for any $i$ then the block is not RoCK, and if this inequality does not hold for any $i$, the block is RoCK.    
\end{enumerate}
This algorithm sounds laborious, and indeed it is if it is done by hand; however, it is extremely efficient to when done by a computer, especially if compared to any algorithm that requires enumeration of the partitions in a block.   In particular, its complexity does not increase significantly as we add $e$-hooks to a block, whereas enumerating partitions will get much worse.  This program is publicly available for Sage \href{https://cocalc.com/share/public_paths/684c1deddca6be23d9d830298743118c07c493b6}{here on CoCalc}, and some examples are shown later in the paper.    

Let us now consider an example with $e=3$.  Let $\Lambda=2\Lambda_1+\Lambda_2+\Lambda_3$ and $\mu=\Lambda-2\al_1-2\al_2-3\al_3$.  Note that in this case, we have $W_{\mu}=\{1,s_3\}$.  
In this case, $k_{12}^-=2,k_{13}^-=3,k_{23}^-=2,k_{12}^+=1,k_{13}^+=1,k_{23}^+=1$.  Thus, the Scopes walls are given by the red hyperplanes below; the fundamental alcove is colored green, and the block closest to the fundamental in each RoCK chamber is colored red if they are in the positive Weyl chamber for $W_{\mu}$ (the one containing the fundamental alcove), and blue if they are in the negative Weyl chamber.  Thus, there are three RoCK equivalence classes of blocks under Scopes equivalence, corresponding to the three red chambers.
\begin{equation}\label{eq:sl3}
\tikz[very thick,scale=1.1,baseline]{
  \fill[blue!50!white] (3,{sqrt(3)}) -- (4,{sqrt(3)}) -- (3.5, {3/2*sqrt(3)})-- cycle;
    \fill[blue!50!white] (.5,{5/2*sqrt(3)}) -- (1.5,{5/2*sqrt(3)}) -- (1, {2*sqrt(3)})-- cycle;
  \fill[red!50!white] (0,{-2*sqrt(3)}) -- (1,{-2*sqrt(3)}) -- (.5, {-3/2*sqrt(3)})-- cycle;
  \fill[blue!50!white] (3.5,{-sqrt(3)/2}) -- (4.5,{-sqrt(3)/2}) -- (4, {-sqrt(3)})-- cycle;
   \fill[red!50!white] (-1.5,{-sqrt(3)/2}) -- (-2.5,{-sqrt(3)/2}) -- (-2, {-sqrt(3)})-- cycle;
  \fill[red!50!white] (-3,{sqrt(3)}) -- (-2,{sqrt(3)}) -- (-2.5, {3/2*sqrt(3)})-- cycle;
    
    \fill[green!50!white] (1,0) -- (1/2,{sqrt(3)/2})--(0,0)-- cycle;
\foreach \k in {-1,0,1,2} {
  \draw[red] (5,{\k*sqrt(3)/2})--(-4.5,{\k*sqrt(3)/2});  }
\foreach \k in {-5,-4,-3,-2,3,4,5} {
  \draw[gray!50!white] (5,{\k*sqrt(3)/2})--(-4.5,{\k*sqrt(3)/2});  }
  \foreach \k in {-1,0,1,2} {
 \draw[red] (2.5+\k,{5/2*sqrt(3)})-- (-2.5+\k,-{5/2*sqrt(3)});  }
  \draw[gray!50!white] (.5,{5/2*sqrt(3)})-- (-4.5,-{5/2*sqrt(3)});  
  \draw[gray!50!white] (-.5,{5/2*sqrt(3)})-- (-4.5,-{1.5*sqrt(3)});  
  \draw[gray!50!white] (-1.5,{5/2*sqrt(3)})-- (-4.5,-{.5*sqrt(3)});  
  \draw[gray!50!white] (-2.5,{5/2*sqrt(3)})-- (-4.5,{.5*sqrt(3)});  
  \draw[gray!50!white] (-3.5,{5/2*sqrt(3)})-- (-4.5,{1.5*sqrt(3)});  
  \draw[gray!50!white] (.5,{-5/2*sqrt(3)})-- (5,{2*sqrt(3)}); 
  \draw[gray!50!white] (1.5,{-5/2*sqrt(3)})-- (5,{sqrt(3)});  
  \draw[gray!50!white] (2.5,{-5/2*sqrt(3)})-- (5,0);  
  \draw[gray!50!white] (3.5,{-5/2*sqrt(3)})-- (5,{-sqrt(3)});  
  \draw[gray!50!white] (4.5,{-5/2*sqrt(3)})-- (5,{-2*sqrt(3)}); 
   \foreach \k in {-1,0,1,2} {  \draw[red] (-2.5+\k,{5/2*sqrt(3)})-- (2.5+\k,-{5/2*sqrt(3)});}
   \draw[red] (.5,{5/2*sqrt(3)})-- (5,-{2*sqrt(3)});  
 \draw[gray!50!white] (.5,-{5/2*sqrt(3)})-- (-4.5,{5/2*sqrt(3)});
  \draw[gray!50!white] (-.5,{-5/2*sqrt(3)})-- (-4.5,{1.5*sqrt(3)});  
  \draw[gray!50!white] (-1.5,{-5/2*sqrt(3)})-- (-4.5,{.5*sqrt(3)});  
  \draw[gray!50!white] (-2.5,{-5/2*sqrt(3)})-- (-4.5,{-.5*sqrt(3)});  
  \draw[gray!50!white] (-3.5,{-5/2*sqrt(3)})-- (-4.5,{-1.5*sqrt(3)});   
  \draw[gray!50!white] (1.5,{5/2*sqrt(3)})-- (5,{-sqrt(3)});  
  \draw[gray!50!white] (2.5,{5/2*sqrt(3)})-- (5,0);  
  \draw[gray!50!white] (3.5,{5/2*sqrt(3)})-- (5,{sqrt(3)});  
  \draw[gray!50!white] (4.5,{5/2*sqrt(3)})-- (5,{2*sqrt(3)});
  \draw[blue,line width=.3em] (-1.5,{5/2*sqrt(3)})-- (3.5,-{5/2*sqrt(3)});
 }  
\end{equation}  

\section{Examples}\label{sec:examples}
\subsection{Level 1}

Let $H_n(\K,q)$ be the Hecke algebra of $S_n$ (i.e. type A) over the field $\K$ for a scalar $q\in \K$, and let $e$ be the quantum characteristic, that is, the smallest integer such that $q^{e-1}+\cdots + 1=0$.  Let us emphasize that we include the case of $q=1$, in which case the quantum characteristic is the usual characteristic of the field $\K$.  In this case, we will explain how we recover the original definition of RoCK blocks (those that satisfy the conditions of \cite[Th. 2]{chuangSymmetric2002}

Let $\fg=\slehat$.  
The category $\bigoplus_{n\geq 0}H_n(\K,q)\mmod$ is a categorical $\fg$-module, categorifying the simple module with the highest weight $\Lambda_e=(1,0,\dots, 0)$.   The action of $\eE_i$ is by $i$-restriction and $\eF_i$ by $i$-induction.   We can assign a weight to each partition $\nu$ by letting $b_i$ be the number of boxes of content $i\pmod e$, and considering 
\[\mu_\nu=\Lambda_e -\sum_{i=1}^{e} b_i\al_i=(1, b_{e}-b_1, b_1-b_2,\dots, b_{e-1}-b_{e},b_e).\] In particular, $\mu_{\emptyset}=(1,0,\dots, 0)$ is the highest weight appearing. Note that the addition of an $e$-hook only changes this weight by adding $\delta$, so the image of $\mu_{\nu}$ in $\fh^*_1$ is determined just by the core of $\mu$, and there will be a unique $e$-core for each integral point of $\fh^*_1$. For example, if $e=2$, then $\fh^*_1$ is parameterized by the value of $h_1$ as before, and the $2$-cores index the points in the $W$-orbit of $h_1=0$.
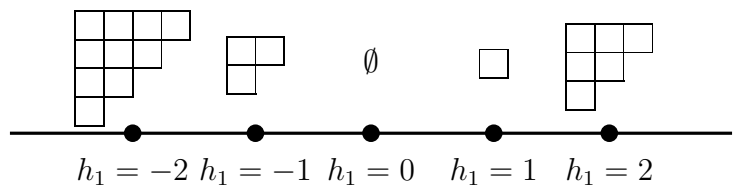
\[\tikz[very thick,scale=1.6]{\draw (3,0)-- node[pos=.33,above,scale=.6,yshift=30pt]{$\ydiagram{1}$} node[circle, fill, pos=.33,scale=.6]{} node[pos=.5,above,yshift=18pt]{$\emptyset$}  node[circle, fill, pos=.5,scale=.6]{} node[pos=.66,above,scale=.6,yshift=20pt]{$\ydiagram{2,1}$} node[pos=.17,above,scale=.6,yshift=10pt]{$\ydiagram{3,2,1}$} node[circle, fill, pos=.17,scale=.6]{} node[circle, fill, pos=.66,scale=.6]{} node[circle, fill, pos=.83,scale=.6]{} node[pos=.83,above,scale=.6]{$\ydiagram{4,3,2,1}$} 
node[pos=.33,below,yshift=-5pt]{$h_1=1$} 
node[pos=.5,below,yshift=-5pt]{$h_1=0$} 
node[pos=.17,below,yshift=-5pt]{$h_1=2$} 
node[pos=.66,below,yshift=-5pt]{$h_1=-1$} 
node[pos=.83,below,yshift=-5pt]{$h_1=-2$} 
(-3,0);
}\]  

\ytableausetup{boxsize=.8em}
Similarly, 3-cores index the integral points in $\fh_1^*$ for $\widehat{\mathfrak{sl}}_3$.  We mark the point $(1,a,b,c,*)$ with $(a,b,c)$.  
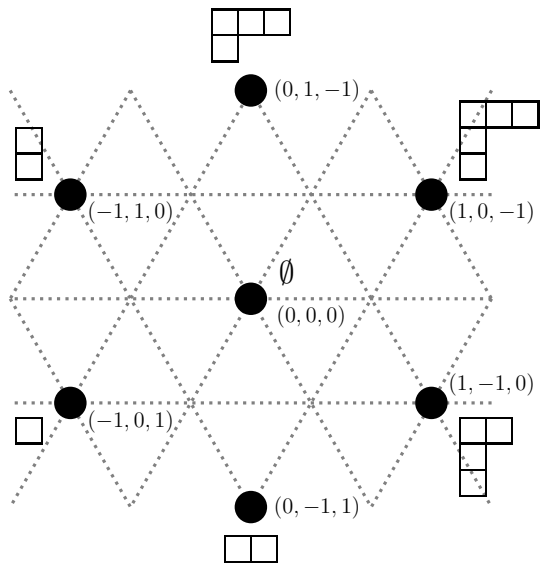
  \[\tikz[very thick,scale=1.6]{
  \draw[black!50!white,dotted] (2,0)--(-2,0); \draw[black!50!white,dotted] (2,{sqrt(3)/2})--(-2,{sqrt(3)/2});  \draw[black!50!white,dotted] (2,{-sqrt(3)/2})--(-2,{-sqrt(3)/2});\draw[black!50!white,dotted] (1,{sqrt(3)})-- (-1,-{sqrt(3)}); \draw[black!50!white,dotted] (-1,{sqrt(3)})-- (1,-{sqrt(3)});
  \draw[black!50!white,dotted] (2,{sqrt(3)})-- (0,-{sqrt(3)}); \draw[black!50!white,dotted] (-2,{sqrt(3)})-- (0,-{sqrt(3)}); \draw[black!50!white,dotted] (0,{sqrt(3)})-- (2,-{sqrt(3)}); \draw[black!50!white,dotted] (0,{sqrt(3)})-- (-2,-{sqrt(3)});
    \draw[black!50!white,dotted] (2,0)-- (1,-{sqrt(3)}); \draw[black!50!white,dotted] (-2,0)-- (-1,-{sqrt(3)}); \draw[black!50!white,dotted] (1,{sqrt(3)})-- (2,0); \draw[black!50!white,dotted] (-1,{sqrt(3)})-- (-2,0); \node[draw=black,circle,fill,label=8:{$\ydiagram{3,1,1}$}] at (3/2, {sqrt(3)/2}){}; \node[scale=.7] at ({3/2+.5}, {sqrt(3)/2-.15}){$(1,0,-1)$};\node[draw=black,circle,fill,label=188:{$\ydiagram{1}$}] at (-3/2, {-sqrt(3)/2}){}; \node[scale=.7] at ({-3/2+.5}, {-sqrt(3)/2-.15}){$(-1,0,1)$};
    \node[draw=black,circle,fill,label=-8:{$\ydiagram{2,1,1}$}] at (3/2, {-sqrt(3)/2}){}; \node[scale=.7] at ({3/2+.5}, {-sqrt(3)/2+.15}){$(1,-1,0)$};\node[draw=black,circle,fill,label=-188:{$\ydiagram{1,1}$}] at (-3/2, {sqrt(3)/2}){};\node[scale=.7] at ({-3/2+.5}, {sqrt(3)/2-.15}){$(-1,1,0)$};
      \node[draw=black,circle,fill,label=below:{$\ydiagram{2}$}] at (0, {-sqrt(3)}){};     \node[scale=.7] at (.55, {-sqrt(3)}){$(0,-1,1)$};\node[draw=black,circle,fill,label=above:{$\ydiagram{3,1}$}] at (0, {sqrt(3)}){};\node[scale=.7] at (.55, {sqrt(3)}){$(0,1,-1)$};
      \node[draw=black,circle,fill,label=8:{$\emptyset$}] at (0,0){};\node[scale=.7] at (.5,-.15) {$(0,0,0)$};}\]
The weight $\mu_{\nu}$ is dominant if and only if the core is empty since no other integral weight is in the fundamental alcove.  That is, if $\mu_{\nu}=\Lambda_e-k\delta=(1,0,\dots, 0,k)$ for $k\geq 0$, the weight of the corresponding block.  The corresponding category is the principal block of $H_{ke}(\K,q)$.  

The other prominent example of another categorical action with the same set of nonzero weights is the module categories over the $q$-Schur algebra appearing endomorphisms of the permutation modules over $H_n(\K,q)$.  The blocks of this algebra are in obvious bijection with those of $H_n(\K,q)$, and the category of modules has an induced categorical action.  Thus, the Scopes chambers are the same, and the RoCK blocks of the Schur algebra are the same as those of the Hecke algebra.

Since any non-empty Young diagram has a box of content $0$, the category $\cC_{\Lambda_e-\alpha_{ij;m}}$ is trivial unless $m>0$; on the other hand, if $m>0$ we can find a non-trivial object on this category given by a hook with removable boxes of content $i$ and $j-1$.  Thus, 
the set $N$ of roots such that $\cC_{\mu+\alpha}$ is thus given by $\alpha_{ij;n}$ for $n<k$, since 
\[\mu_\nu+\alpha_{ij;n}=\mu_{\emptyset}-\alpha_{ij;k-n}.\] 
Thus, by Lemma \ref{lem:Scopes-walls}, the Scopes chambers are cut out by the hyperplanes
\[h_i-h_j=nc \qquad n=k-1,k-2,\dots, -k+1.\] 
Note that we could also derive this from Lemma \ref{lem:simple-support}: $\Lambda_e-k\delta+\al_{ij;m}$ is never dominant, and the dominant element in its Weyl group orbit is $\Lambda_e-(k-m-1)\delta$, so as desired, we must have $m<k$.

As discussed previously, we can describe all weights in this case in the form $w^{-1}\mu_{k}$ for $w$ a shortest right coset representative of $W_\la=\bar{W}$.
In the dominant finite Weyl chamber, there is a single RoCK Scopes chamber given by the elements such that $h_i\geq h_{i+1}+k-1$.  This is the usual RoCK condition on blocks (for example, introduced at the start of \cite[\S 3]{chuangSymmetric2002}).

We can visualize the Scopes chambers on blocks by drawing the corresponding core of weight $w\cdot \mu_{\emptyset}$ over the alcove $wA$ for $w$ a shortest right coset representative.  Our $e=2$  example then becomes: \ytableausetup{boxsize=normal}
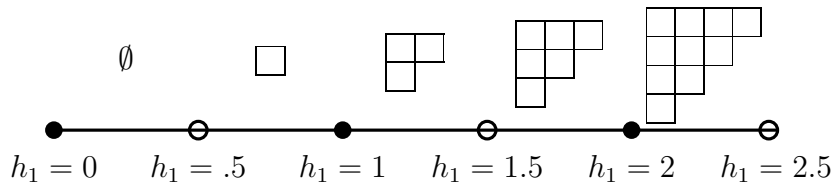
\[\tikz[very thick,scale=1.6]{\draw (3,0)--  node[pos=.9,above,yshift=18pt]{$\emptyset$}  node[pos=.7,above,scale=.6,yshift=30pt]{$\ydiagram{1}$} node[pos=.5,above,scale=.6,yshift=20pt]{$\ydiagram{2,1}$} node[pos=.1,above,scale=.6]{$\ydiagram{4,3,2,1}$}  node[pos=.3,above,scale=.6,yshift=10pt]{$\ydiagram{3,2,1}$} node[circle, fill, pos=.2,scale=.6]{} node[circle, fill, pos=.6,scale=.6]{} node[circle, draw, pos=.8,scale=.6]{} node[circle, draw, pos=.4,scale=.6]{} node[circle, draw, pos=0.01,scale=.6]{} node[circle, fill, pos=1,scale=.6]{}
node[pos=.4,below,yshift=-5pt]{$h_1=1.5$} 
node[pos=.6,below,yshift=-5pt]{$h_1=1$} 
node[pos=.2,below,yshift=-5pt]{$h_1=2$} 
node[pos=0,below,yshift=-5pt]{$h_1=2.5$}
node[pos=.8,below,yshift=-5pt]{$h_1=.5$} 
node[pos=1,below,yshift=-5pt]{$h_1=0$} 
(-3,0);
}\]  
We've drawn walls of the form $h_1\in \Z$ with solid dots and $h_1\in \Z+\frac{1}{2}$ with open dots.  For a fixed $k$, the first $k-1$ alcoves are each a separate Scopes class, and then all the others are RoCK.
Our $e=3$ example becomes:
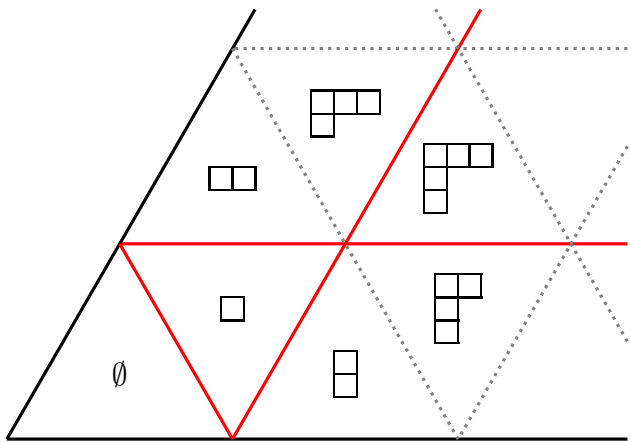
\ytableausetup{boxsize=.7em}
  \[\tikz[very thick,scale=3]{
  \draw[black] ({2.5*1.1},0)--(0,0); \draw[red] ({2.5*1.1},{sqrt(3)/2})--(1/2,{sqrt(3)/2});  \draw[black!50!white,dotted] ({2.5*1.1},{sqrt(3)})--(1,{sqrt(3)}); \draw (1.1,{sqrt(3)*1.1})-- (0,0);   \draw[red] (2.1,{sqrt(3)*1.1})-- (1,0);  \draw[red] (.5,{sqrt(3)/2})-- (1,0);  \draw[black!50!white,dotted] (1,{sqrt(3)})-- (2,0);\draw[black!50!white,dotted] ({2.5*1.1},{sqrt(3)*3/4})-- (2,0);\draw[black!50!white,dotted] ({2.5*1.1},{sqrt(3)*1/4})-- (1.9,{sqrt(3)*1.1});
  \node at (3/2,{sqrt(3)*5/6}) {$\ydiagram{3,1}$}; 
  \node at (2,{sqrt(3)*4/6}) {$\ydiagram{3,1,1}$};\node at (1, {sqrt(3)/3}){$\ydiagram{1}$};
    \node at (2,{sqrt(3)/3}) {$\ydiagram{2,1,1}$};\node at (3/2, {sqrt(3)/6}){$\ydiagram{1,1}$};
      \node at (1, {sqrt(3)*4/6}){$\ydiagram{2}$};
      \node at (.5,{sqrt(3)/6}){$\emptyset$};}\]
 We've drawn in the walls that separate $k=2$ Scopes chambers in red.  As we increase $k$, we add more and more translates of each of these hyperplanes.

 \subsection{Ariki-Koike algebras}
 
 Now, we turn to the more interesting case of Ariki-Koike algebras of level $\ell$. Throughout this section, we will assume that $\mathfrak{g}=\slehat$ and $\mathcal{C}$ carries a categorical action whose support coincides with that of the simple $V_{\Lambda}$ for $\Lambda$ a dominant weight of level $\ell$. 
 
  \subsubsection{Background} The motivating example for us is the Ariki-Koike algebra $AK_{n}(\K,q,\Bw)$ associated to the polynomial 
 \[f(u)=\prod_{i\in \Z/e\Z} (u-q^i)^{w_i}\]
 for fixed $q\in \K\setminus \{0\}$ of multiplicative order $e>1$, $w_i\in \Z_{\geq 0}$ with $\sum w_i=\ell$.  
 
% The categories $\oplus_{n}AK_{n}(\K,q,\Bw)\mmod$ carry a categorical action of $\slehat$.  
It was proven by Ariki
\cite[Prop. 4.5]{arikiDecompositionNumbers1996}
 that the Grothendieck group of this category is isomorphic to the highest weight simple $V_{\Lambda}$ for $\Lambda=\sum_{i=1}^e w_i\Lambda_i\in \fh_{\ell}^*$ where \[\Lambda_i=(1,\underbrace{\frac{e-i}{e},\dots, \frac{e-i}{e}}_{i\text{ times}},\underbrace{-\frac{i}{e},\dots, -\frac{i}{e}}_{e-i\text{ times}},0).\]   
 These are fundamental weights since $\al_i^{\vee}(\Lambda_j)=\delta_{ij}$.  A typical $\Lambda$ will have all $w_i\geq 0$ and thus will be in the interior of the fundamental alcove and have trivial stabilizer.  These are a bit easier to write if we adopt the convention that we can denote weights by arbitrary elements of $\R^{n+2}$, which we take to be equivalent to their orthogonal projection to the subspace \[\fh=\{(h_0,h_1\dots,h_{e+1})\in  \R^{e+2}\mid h_1+\cdots+h_e=0\}.\] In this case, we can equally well write $\Lambda_i=(1,\underbrace{1,\dots, 1}_{i\text{ times}},\underbrace{0,\dots, 0}_{e-i\text{ times}},0)$.
 
Ariki's action on the Grothendieck group reflects a categorical action of $\slehat$ on $\oplus_{n}AK_{n}(\K,q,\Bw)\mmod$; the fact that $i$-induction and $i$-restriction for a given $i$ define a strong $\mathfrak{sl}_2$ action is proven in \cite[\S 7.2.2]{chuangDerivedEquivalences2008}.  The fact that this extends to a categorical action of $\slehat$ is effectively equivalent to the main theorem of \cite{BKKL}. We can more systematically construct the categorical action by $i$-induction and $i$-restriction functors using the formalism of quantum Heisenberg actions.  Such an action on Ariki-Koike algebras is defined in \cite[\S 6]{Brundandef}.   The main theorem of \cite{brundanHeisenbergKacMoody2020} implies that this gives us a categorical $\slehat$-action.  The theorem of Ariki above shows that it satisfies our assumptions from the start of the section.

We can apply the same principle to construct categorical actions of $\slehat$ on \begin{enumerate}
  \item cyclotomic degenerate affine Hecke algebras in characteristic $e$, assuming $e$ is prime;
  \item cyclotomic $q$-Schur algebras;
  \item categories $\cO$ of the rational Cherednik algebra of $G(\ell,1,n)$ with $\kappa=a/e$ as a fraction in least terms.
\end{enumerate}
In case (1), we have a degenerate Heisenberg action by \cite[Th. 6.7]{mackaayDegenerate2018}; in case (3), the quantum Heisenberg action is constructed in \cite[\S 7]{Brundandef}, and case (2) is a limiting case of (3) (see, for example, \cite[Cor. 3.11]{WebRou}).  The categorical action in case (3) was first constructed by Shan \cite{ShanCrystal}, but for an unnecessarily restrictive choice of parameters and using a different formalism for categorical actions.

All of these also have support equal to that of an irreducible representation of the form $V(\Lambda)$:
\begin{enumerate}
  \item[(1)] The weight $\Lambda$ is determined by the roots of the cyclotomic polynomial, by identifying the elements of $\Fp$ with the Dynkin diagram of $\slphat$, as in \cite{BKKL}.
  \item[(2-3)] Each block in these cases is a quasi-hereditary cover of a block of the Ariki-Koike algebra, so we simply use the same weight combinatorics. In particular, we obtain a nonzero block if and only if the corresponding block of Ariki-Koike is nonzero.  
\end{enumerate} 
Everything we will say below also applies to other categorifications of $\slehat$ with the same support.

 The blocks of $ AK_{n}(\K,q,\Bw)$ again correspond to weight spaces: by \cite{lyleBlocks2007}, the block of Specht module is determined by the number of $b_i$ boxes of residue $i\pmod e$ in the corresponding charged multipartition, and this also determines the weight by the formula:
\[\mu_\nu=\Lambda-\sum_{i=1}^{e-1} b_i\al_i.\]

When we associate an $\ell$-tuple of runner diagrams to a simple as in \cite[\S 2.2]{lyleRouquier2022} 
the total number $t_i$ of beads on the $i$th runners of the $\ell$ different abacus diagrams is given by $t_i=b_{i-1}-b_i+\sum_{j=i}^{e}w_i$ for $i=1,\dots, e$.  This also allows us to write the weight $\mu_\nu$ in the coordinates $h_i$:
\begin{equation*}
  \mu_{\nu}=(\ell,t_1,\dots, t_{e} , b_e).
\end{equation*}
These coordinates help identify the dominant weight in a given orbit: a weight is dominant if $t_1\geq \cdots\geq t_e\geq t_1-\ell$.   In level 1, these were easy to identify as the principal blocks for ranks divisible by $e$ (i.e., the lowest rank where a given defect group appears).  On the other hand, in higher levels, there are many more dominant weights.

As mentioned in the introduction, \cite[Conj. 1]{lyleRouquier2022} is effectively equivalent to Lemma \ref{lem:Chuang-Rouquier} in this context: since $E_i$ acts by pushing beads from the $(i+1)$st runner to the $i$th, the weight $w^{-1}\la +\al_i$ is not in the support of $\cC$ if and only if it is never possible to push a bead in this way.  This is only the case for an abacus diagram where the weight of the $i$th runner is less than the difference between the number of beads on the $i$th and $(i+1)$st runners; this is precisely the hypothesis of \cite[Conj. 1]{lyleRouquier2022}.

\excise{For each integral point in the fundamental alcove of $\fh_{\ell}^*$, there is a dominant block of minimal rank, whose weight $\mu$ is uniquely characterized by the fact that $\mu+\delta$ is not a weight of the representation $V_{\Lambda}$.   In terms of multipartitions, this will happen if there is no multipartition from which we can remove $e$ boxes, one of each content.  In the level 1 case, it was enough to check if we could remove an $e$-hook from any partition in the block, but now we have to consider all multipartitions in a given block, and allow ourselves to remove some boxes from each of the different components, which together give the contents of a single $e$-hook.  

Having found this $\mu$, the weights sitting above the same point are exactly those of the form $\mu-k\delta$ for $k\geq 0$.  To see this, just note that if we have a multipartition of a given weight, we can always add an $e$-hook to obtain one with the same weight minus $\delta$.  

Finding $N_{\mu}$ is simply a question of precisely which numbers of boxes of a given content can actually be realized by a multipartition, or put differently by which weights $\Lambda-\sum_{i=0}^{e-1}b_i\al_i$ lie in the Weyl polytope of the highest weight $\Lambda$.} 

\subsubsection{Finding RoCK blocks} To test whether a weight is RoCK, we simply apply the algorithm of Section \ref{sec:irreducible} to our weight.  In particular, we have to calculate the statistics $k^{\pm}$; since we assume that $\fg=\slehat$, we can encode these as
  \[k^+_{ij}=k_{ji}^{-}=\begin{cases}
      \max(\{-1\}\cup \{ n \mid \al_{ij;n}\in N_{\mu}\}) &  i<j\\
     \max(\{0\}\cup \{ n \mid \al_{ij;n}\in N_{\mu}\}) & j>i
  \end{cases}.\] 
    Thus, the Scopes walls for $N_{\mu}$ will be of the form 
    \begin{equation}
    h_i-h_j =nc \qquad \qquad n=[-k^+_{ij},k^-_{ij}].\label{eq:Scopes-walls}
    \end{equation}  

For readers more comfortable with multipartitions and abacus diagrams, let us reinterpret these statistics, in terms of the abacus diagrams of all multipartitions of weight $\mu$. 
\begin{lemma}
Assume that for some integer $n\geq 0$, there is a multipartition of $\mu$ such that there is a bead on the $i$th runner of its abacus diagram and an empty space on the $j$th runner $n$ rows higher.  The statistic $k^+_{ij}=k^-_{ji}$ is the largest integer $n$ such that this property holds.
\end{lemma}
\begin{proof}
    Consider the element $e_{ji}\otimes t^n\in \fg$ (that is, the matrix with $t^n$ as the lone nonzero entry in the $i$th column and $j$th row).  This element acts on the level 1 Fock space representation by sending the standard vector for a partition to the sum of all ways of moving a bead from the $i$th runner to the $j$th runner $n$ rows higher.  Thus, it acts on the level $\ell$ Fock space by the sum of all the ways of doing this to each component of the multipartition.  Let us emphasize that this does not move beads between the abaci for the different components.  
    
    Thus, $\al_{i,j;n}\in N_{\mu}$ if and only if such a runner move is possible.  This completes the proof.   
\end{proof}

    Let us discuss the $e=2$ case.  In this case, $\Lambda=w_1\Lambda_1+w_2\Lambda_2$.  The weight $\mu=\Lambda-b_1\alpha_1-b_2\alpha_2=(\ell,b_2-b_1+\frac{w_1}{2},b_1-b_2-\frac{w_1}{2}, b_2)$  will be dominant if $w_1/2\geq b_1-b_2\geq -w_2/2$.  In all these cases, the corresponding category is nonzero.  
    
    The condition that $\al_{ij;n}\in N_{\mu}\Rightarrow \mu+\al_{ij;n}\geq \Lambda$ translates into upper bounds:
    \begin{equation}\label{eq:12-bounds}
      k_{12}^+\leq b_{12}=\min(b_1-1,b_2)\qquad k_{21}^+=k_{12}^-\leq b_{21}=\min(b_1+1,b_2).
    \end{equation} 
    
    If $\mu+\al_{12;n}$ or $\mu+\al_{21;n}$ is dominant, then Lemma \ref{lem:simple-support} shows that equality holds in \eqref{eq:12-bounds}.  On the other hand, if $\mu+\al_{12;n}=\Lambda-(b_1-n-1)\al_1+(b_2-n)\al_2$ is not dominant, we must have $\al_2^{\vee}(\mu)=w_2-2b_2+2b_1\in \{0,1\}$.    In this case, the dominant element of its orbit is 
\[s_2(\mu+\al_{12;n})=\begin{cases}
  \Lambda-(b_1-n-1)\al_1-(b_2-n-1)\al_2 & w_2-2b_2+2b_1=1\\
  \Lambda-(b_1-n-1)\al_1-(b_2-n-2)\al_2 & w_2-2b_2+2b_1=0
\end{cases}\]
Similarly, we have:
\[s_1(\mu+\al_{21;n})=\begin{cases}
  \Lambda-(b_1-n)\al_1-(b_2-n)\al_2 & w_1-2b_1+2b_2=1\\
  \Lambda-(b_1-n-1)\al_1-(b_2-n)\al_2 & w_1-2b_1+2b_2=0
\end{cases}\]
  \begin{align*}
  k_{12}^+&=\begin{cases}
  b_{12}=\min(b_1-1,b_2) & w_2 >0 \text{ or }b_1>b_2\\
  b_{12}-1=\min(b_1-1,b_2-2) & w_2 =b_1-b_2=0
\end{cases}\\
k^-_{12}&=\begin{cases}
  b_{21}=\min(b_1+1,b_2) & w_1 >0 \text{ or }b_2>b_1\\
  b_{21}-1=\min(b_1-1,b_2) & w_1 =b_2-b_1=0
\end{cases}
\end{align*}.  
\begin{proposition}\label{prop:sl2hat}
The possible pairs of $(k^-_{12},k^+_{12})$ are 
\begin{itemize}
  \item $(n,n)$, achieved if  $b_2=n < b_1$ or $b_1=b_2=n+1$ and $w_1=0$;
  \item $(n,n-1)$, achieved if $b_1=b_2=n$ and $w_2\neq 0,w_1\neq 0$.
  \item $(n+1,n-1)$, achieved if $b_1=n<b_2$ or $b_1=b_2=n+1$ and $w_2=0$;  
\end{itemize}
\end{proposition}
%To summarize, the hyperplanes dividing Scopes chambers are 
%\[h_1-h_2=mc \qquad -k_{12}^+\leq m\leq k^-_{12}. \]
Thus, the minimal RoCK block for the dominant Weyl chamber is given by the weight $(s_1s_2)^{k^-_{12}/2}\mu$ if $k^-_{12}$ is even and $s_2(s_1s_2)^{(k^-_{12}-1)/2}\mu$ if $k^-_{12}$ is odd; similarly, for the anti-dominant chamber, we obtain $(s_2s_1)^{k^+_{12}/2}\mu$ if $k^+_{12}$ is even and $s_1(s_2s_1)^{(k^+_{12}-1)/2}\mu$ if $k^+_{12}$ is odd.

  While interesting as an illustration of our approach, the $\mathfrak{\widehat{sl}}_2$ case also has important consequences for the general $\slehat$ case.
  \begin{proposition}
    For any $i<j$ and $\mu$ dominant: 
  \begin{enumerate}
      \item   The bounds $k_{ij}^{\pm}$ satisfy $k_{ij}^-\geq k_{ij}^+\geq k_{ij}^--2$.
      \item   The corresponding invariants $\bar{k}_{ij}^{\pm}$ for the weight $\mu+\delta$ satisfy 
      \[\bar{k}_{ij}^{+}=\max(\{{k}_{ij}^{+}-1,-1\})\qquad\bar{k}_{ij}^{-}=\max(\{{k}_{ij}^{-}-1,0\})  \]
  \end{enumerate}
  \end{proposition}
  \begin{proof}
    The root $\mathfrak{sl}_2$'s for $\alpha_{ij;0}$ and $\alpha_{ji;1}$ generate a copy of $\widehat{\mathfrak{sl}}_2$.  Any vector in $V_{\Lambda}$ generates a finite direct sum of simple integrable modules of level $\ell$ over this $\widehat{\mathfrak{sl}}_2$, so we can reduce to showing the result in this case.  This follows from Proposition \ref{prop:sl2hat}.
  \end{proof}

  While applying the full algorithm is more precise, using simpler methods, we can give an upper bound on the set of Scopes walls that is easy to compute by hand.  
  To understand this more concretely, it is useful to note that if $n=0$, then $\alpha_{ij}=\al_i+\cdots \al_{j-1}$ and if $n>0$, then \[\alpha_{ij;n}=\al_i+\cdots \al_{e-1}+\al_e+\al_1+\cdots +\al_{j-1}+(n-1)\delta.\]
  That is: 
  \begin{lemma}\label{lem:number-boxes}
     The difference between the number $b_r'$ of boxes  of content $r$ in a multipartition of weight $\mu+\al_{ij;n}=\Lambda -\sum_{i=0}^{e-1}b_i'\al_i$ compared to the number $b_r$ in one of weight $\mu$ is:
     \[b_r-b_r'=\begin{cases}
      n+1 & i<j, r\in [i,j-1] \\
      n-1 & j<i, r\in [j,i-1]\\
      n & \text{otherwise}
     \end{cases}\]  \end{lemma}
%This means that just from the number of boxes in a multipartition of weight $\mu$, we can immediately read off a constraint on the set $N_{\mu}$. 
To phrase this, it is useful to think about the statistic 
\[b_{ij}=\begin{cases}
   \min(\{b_r-1\}_{r\in [i,j-1]} \cup \{b_r\}_{r\notin [i,j-1]}) & i<j\\
      \min(\{b_r+1\}_{r\in [j,i-1]} \cup \{b_r\}_{r\notin [j,i-1]}) & j<i\\
\end{cases}\]
%Let $B=\max(b_{ij})=\max(b_r+1)$; note that $b_{ij}\geq B-2$ for all $i,j$. 
We will also include the restatement based on the fact that the root $\al_{ij;n}\in N_{\mu}$ if and only if the wall $h_i-h_j=nc$ separates two Scopes chambers.
\begin{corollary}\label{cor:bound-1}
We have a bound $k_{ij}^+\leq b_{ij}$, with equality if $\mu+\al_{ij;n}$ is dominant for all $i,j,n$.  
\end{corollary}
\begin{proof}
  We have $n>b_{ij}$ if and only if $\mu+\al_{ij;n}\not \leq \Lambda$, so the corresponding weight space is 0.  If $\mu+\al_{ij;n}$ is dominant, we conversely have that the weight space is not 0 if and only if $\mu+\al_{ij;n}\leq \Lambda$.  
\end{proof}
The dominance condition is needed, as the level 1 case and Proposition \ref{prop:sl2hat} show.  

%Let us now phrase this in terms of runners: adding $\al_{ij;n}$ to $\mu$ corresponds to removing a bead from the $j$th runner of one abacus diagram, and moving it to the $i$th runner of the same diagram,  $n$ rows higher.  Alternatively, it can correspond to a chain of such moves, moving a bead from the $i$th runner to the $i_2$th in one diagram, from the $i_2$ to the $i_3$rd in a potentially different diagram, and so on, until eventually a bead stops on the $j$th runner, and the sum of the number of rows up these beads move sums to $n$.  The constraint of Corollary \ref{cor:bound-1} is the restriction imposed on $n$ by that we can only push a bead from the $(r+1)$st runner to the $r$th $b_r$ times.  

%We see immediately from Corollary \ref{cor:bound-1}:
%\begin{corollary}\label{cor:bound-RoCK}
% All alcoves satisfying $h_i-h_j>b_{ij}^*$ or $h_{i}-h_j<-b_{ji}^*$ for all $i,j$ are RoCK, corresponding to the Weyl chamber induced by the order of $h_i$'s.  
%\end{corollary}
\subsubsection{Stretching and comparison with Rouquier blocks}
The description \eqref{eq:Scopes-walls} allows us to easily identify a RoCK alcove for each Weyl chamber: fix a dominant weight $\mu$  and a permutation $\sigma$, and consider the alcove where 
\[h_{\sigma^{-1}(i)}-h_{\sigma^{-1}(i+1)}\geq u_i \qquad \qquad i=1,\dots e-1\]
\[h_{\sigma^{-1}(e)}-h_{\sigma^{-1}(1)}\geq 1- \sum_{i=1}^{i-1} u_i \]
where $u_i=k^-_{\sigma^{-1}(i),\sigma^{-1}(i+1)}$.  
This is the image of the fundamental alcove under the action of translation by the vector 
\[v_{\sigma}=(\ell \sum_{i=1}^{e-1}u_i, \ell\sum_{i=2}^{e-1}u_i, \dots, 0)\]
followed by the permutation $\sigma^{-1}$.  Note that this might not be an element of the affine Weyl group, since the sum of the entries of $v_{\sigma}$ may not be divisible by $e$; one can correct this by precomposing with an appropriate power of the element of the extended affine Weyl group $\upsilon(h_1,\dots, h_e)=(h_e+1,h_1,\dots, h_{e-1})$.  Note that since we are considering the projection of these vectors onto the subspace $h_1+\dots+ h_e=0$, the $e$th power of this transformation is trivial.  This gives us a possibly different permutation $\sigma'$ and $v_{\sigma'}$.  We can avoid this if we only require that $u_i\geq k^-_{\sigma^{-1}(i),\sigma^{-1}(i+1)}$ and choose the $u_i$ so that the entries of $v_{\sigma}$ have sum divisible by $e$, though in this case we will not find the closest RoCK block to the fundamental alcove.  

Some care is needed here, because only the usual affine Weyl group, not the extended one, acts on the weights of an arbitrary representation, so $w^{-1}\mu$ would not make sense for $w$ in the extended affine Weyl group.  That is, if we act by an element of the extended affine Weyl group, the result will correspond to a block of a potentially different Ariki-Koike algebra where we have multiplied the roots by $q^s$ for some $s$.  

Assume that we choose $\mathbf{u}$ so that the sum of the entries of $v_{\sigma}$ is divisible by $e$.  In this case, the block of weight $w^{-1}\mu$  is a  ``stretching'' of the block for $\sigma\mu$. Stretching further can then be achieved by simply increasing the values of $u_i$; for sufficiently large values, this will clearly give a Rouquier block.  It is easy to get mixed up by the fact that we invert $w$: the alcoves $wA$ we consider above are translates of the tip of a Weyl chamber in a direction deep in that chamber, as shown in \eqref{eq:sl3}.  After inverting, this means that we have translated far in one fixed direction to obtain $t_1'<t_2'<\dots< t_e'$ in $(t_1',\dots, t_e')=w^{-1}(t_1,\dots, t_e)$, independent of which Weyl chamber we chose.  
 
 The Weyl chamber is encoded by the permutation $\sigma$, or put differently, by how we must permute the residues of $t_i\pmod\ell$ to match those of $t_i'$.  Note that the stabilizer $W_{\la}$ exactly reflects how many of these residues are the same, that is, the failure of this permutation to be unique.  In Lyle's indexing set, this is reflected in the choice of the vector $\mathbf{a}$ in \cite[\S 3.4]{lyleRouquier2022}.

Finally, let us compare with the definition of Rouquier blocks in the sense of \cite[\S 3.1]{lyleRouquier2022}.  According to this definition, a block is {\bf Rouquier} if for every $\ell$-partition corresponding to a module in the block consists of Rouquier partitions (i.e., those that correspond to an $\ell=1$ RoCK block).  
\begin{proposition}
  Every Rouquier block is a RoCK block in the sense of Definition \ref{def:our-RoCK}, and every RoCK block is Scopes equivalent to a Rouquier block.
\end{proposition}
\begin{proof}
{\bf Stretching a Rouquier block does not change its Scopes class}:  If we stretch a Rouquier block by a vector $\mathbf{M}$ satisfying $M_1\ll M_2\ll \cdots \ll M_e$ with $\sum M_i=0$, then \cite[Th. 3.20 \& 3.21]{lyleRouquier2022} constructs a decomposition equivalence between the original block and the shifted block (which is Rouquier).  This is achieved by studying an element of the extended affine Weyl group that sends the original block to the stretched one, following the recipe of \cite[Lem. 3.8]{lyleRouquier2022}.
For a general choice of $M_i$, this would be the action of an element of the extended affine Weyl group, but the fact that $\sum_iM_i=0$ guarantees that it is a product of reflections in the usual affine Weyl group, i.e. the total number of times we have applied $\upsilon$ sums to 0, and so we can commute them past the reflections in the finite Weyl group to get a reduced word in the affine Weyl group.  Each of these affine reflections satisfies the hypotheses of \cite[Thm. 3.10]{lyleRouquier2022}.  That is, when we compare the two runners we will switch, the left one has so many more beads than the right one that it is impossible to shift a bead leftward, i.e. the corresponding simple is highest weight for this root $\mathfrak{sl}_2$.

Since this holds for all simples in the block, hypothesis (2) of Lemma \ref{lem:Chuang-Rouquier} holds, and we can conclude that we never change the Scopes class of this block.  

{\bf A generic stretching is RoCK}:
If we take any block and stretch it by $M_1\ll M_2\ll \cdots \ll M_e$ for sufficiently large parameters, then it will be RoCK by \eqref{eq:Scopes-walls}.  

{\bf Rouquier $\Rightarrow$ RoCK}: Consider any Rouquier block.  As discussed, we can stretch it by $M_1\ll M_2\ll \cdots \ll M_e$  and preserve the Scopes equivalence class.  This means that its Scopes equivalence class contains a generic stretching and is RoCK.  

{\bf RoCK $\Rightarrow$ Rouquier}: We discussed above how with appropriate choice of $u_i$, we can construct a RoCK block for any Weyl chamber which is Rouquier.  
\end{proof}

\subsubsection{Computer calculations} For higher rank, let us illustrate with an example of how to use \href{https://cocalc.com/share/public_paths/684c1deddca6be23d9d830298743118c07c493b6}{our Sage program}.  We will reproduce the example of \cite[\S 3.1]{lyleRouquier2022}.  In this case, the quantum characteristic is $e=4$ and the multicharge $(2,0)$, and one can calculate that the 2-partition $\boldsymbol{\la}$ for this example is \[((10, 7, 6, 5, 5, 3, 3, 1, 1, 1), (16, 13, 10, 7, 7, 5, 5, 3, 3, 3, 2, 2, 2, 1, 1, 1)).\] We enter this into Sage, and use the existing functionality to find the block for this 2-partition: 
\begin{lstlisting}
sage: e=4  
sage: multicharge=(2,0) 
sage: lam=PartitionTuple([[10, 7, 6, 5, 5, 3, 3, 1, 1, 1], [16, 13, 10, 7, 7, 5, 5, 3, 3, 3, 2, 2, 2, 1, 1, 1]])
sage: lam.block(e,multicharge)                                                                                                                                                        
{2: 34, 3: 32, 0: 25, 1: 32}
\end{lstlisting}
This tells us that $\mu=\Lambda_4+\Lambda_2-32\al_1-34\al_2-32\al_3-25\al_4$ (though note that the values are out of order).  We can now set the block equal to this value and find the corresponding dominant weight.  The first sequence returned is a reduced word of a minimal length $w\in W$ such that $w\mu$ is dominant, so in this case, \[w=s_1s_2s_1s_3s_4s_1s_2s_1s_4s_3s_4s_1s_2s_1s_4s_3s_4s_2s_1s_4s_3s_4s_2.\]
\begin{lstlisting}
sage: al={2: 34, 3: 32, 0: 25, 1: 32}                                                                                                                                                 
sage: find_dominant(al,e,multicharge)                                                                                                                                                 
[[1, 2, 1, 3, 0, 1, 2, 1, 0, 3, 0, 1, 2, 1, 0, 3, 0, 2, 1, 0, 3, 0, 2],
 {2: 3, 3: 3, 0: 3, 1: 3}]
\end{lstlisting} 
Alternatively, we can do step (1) of the algorithm from Section \ref{sec:irreducible} directly using the function  \verb|find_dom_w_chamber|, which returns the $h$-coordinates of a point in $wA$ in addition to the weight $w\mu$:  
\begin{lstlisting}
sage: al={2: 34, 3: 32, 0: 25, 1: 32}                                                                                                                                                 
sage: find_dom_w_chamber(al,e,multicharge)                                                                                                                                            
[[3/2, -3, 15/4, -3/4], {2: 3, 3: 3, 0: 3, 1: 3}]
\end{lstlisting}
The computer can now find the set $N_{\mu}$ using the function \verb|findN|:
\begin{lstlisting}
sage: dom=find_dom_w_chamber(al,e,multicharge)[1]                                                                                                                                     
sage: dom                                                                                                                                                                             
{2: 3, 3: 3, 0: 3, 1: 3}
sage: findN(dom,e,multicharge)                                                                                                                                                        
{(2, 1): 2, (3, 1): 3, (4, 1): 3, (1, 2): 2, (3, 2): 3, 
 (4, 2): 3, (1, 3): 2, (2, 3): 2, (4, 3): 2, (1, 4): 2,
 (2, 4): 2, (3, 4): 2}
\end{lstlisting}
The output here is a dictionary, associating the value $k_{ij}^+=k_{ji}^-$ to the key \verb|(i,j)|.  To find the Scopes chamber of this block, one checks the sign of all the corresponding coroots on the weight $\gamma=(1,3/2, -3, 15/4, -3/4,0)$, or generally on the alcove containing this point, which is defined by the inequalities $h_1\geq h_4+2\geq h_2+4\geq h_3-2\geq h_1-1$.  As we calculated above, we have $\{\al_{41;3},\al_{41;2},\al_{41;1},\al_{14;0},\al_{14;1},\al_{14;2}\}\subset N_{\mu}$, and so we have Scopes walls at 
\[h_1-h_4=m \qquad m=-2,-1,0,1,2,3\]
 This tells us that in the Scopes chamber containing the point $\gamma$, we have $2<h_1-h_4<3$, and this block is not RoCK.  

It is not necessary to perform these computations by hand; the full algorithm is automated by the function \verb|test_RoCK|:
\begin{lstlisting}
sage: al={2: 34, 3: 32, 0: 25, 1: 32}                                                                                                                                                 
sage: test_RoCK(al,e,multicharge)
False
\end{lstlisting}
There is also a version of this function with a verbose output that points out which Scopes walls witness the failure to be RoCK:
\begin{lstlisting}
sage: test_RoCK_verbose(al,e,multicharge)                                                                                                                                             
The pair [1, 2]  is OK since  9/2  is not between  2 and -2.
The pair [1, 3]  is OK since  -9/4  is not between  3 and -2.
The pair [2, 3]  is OK since  -27/4  is not between  3 and -2.
There's a problem with  [1, 4]  since  3 > 9/4 > -2.
The pair [2, 4]  is OK since  -9/4  is not between  3 and -2.
The pair [3, 4]  is OK since  9/2  is not between  2 and -2.
False
\end{lstlisting}
If you want to construct RoCK blocks, then the function \verb|RoCK_weight| takes in an alcove (in the form of a point in the alcove) and a fixed block and outputs the closest alcove in the same Weyl chamber which is RoCK for the orbit of the block we fixed.
\begin{lstlisting}
sage: Ral=RoCK_weight([3/2, -3, 15/4, -3/4],al,e,multicharge)                                                                                                                         
sage: Ral                                                                                                                                                                             
{0: 45, 1: 42, 2: 34, 3: 42}
sage: test_RoCK_verbose(Ral,e,multicharge)                                                                                                                                            
The pair [1, 2]  is OK since  11/2  is not between  2 and -2.
The pair [1, 3]  is OK since  -9/4  is not between  3 and -2.
The pair [2, 3]  is OK since  -31/4  is not between  3 and -2.
The pair [1, 4]  is OK since  13/4  is not between  3 and -2.
The pair [2, 4]  is OK since  -9/4  is not between  3 and -2.
The pair [3, 4]  is OK since  11/2  is not between  2 and -2.
True
sage: weight_from_block(Ral,e,multicharge, 9)                                                                                                                                         
[[13, 18, 1, 6], [10, 10, 9, 9]]
\end{lstlisting}
Note that we used the function \verb|weight_from_block| to find the value of $(t_1,t_2,t_3,t_4)$ for this block.  Finally, we can find an example of a RoCK block for each Weyl chamber:
\begin{lstlisting}
sage: all_RoCKs(al,e,multicharge)                                                                                                                                                     
{(3/4, 1/2, 1/4, 0): {0: 37, 1: 34, 2: 27, 3: 34},
 (3/4, 1/2, 0, 1/4): {0: 37, 1: 34, 2: 27, 3: 34},
 (3/4, 1/4, 1/2, 0): {0: 34, 1: 42, 2: 45, 3: 42},
 (3/4, 1/4, 0, 1/2): {0: 34, 1: 42, 2: 45, 3: 42},
 (3/4, 0, 1/2, 1/4): {0: 37, 1: 29, 2: 37, 3: 39},
 (3/4, 0, 1/4, 1/2): {0: 37, 1: 29, 2: 37, 3: 39},
 (1/2, 3/4, 1/4, 0): {0: 37, 1: 34, 2: 27, 3: 34},
 (1/2, 3/4, 0, 1/4): {0: 37, 1: 34, 2: 27, 3: 34},
 (1/2, 1/4, 3/4, 0): {0: 37, 1: 39, 2: 37, 3: 29},
 (1/2, 1/4, 0, 3/4): {0: 37, 1: 39, 2: 37, 3: 29},
 (1/2, 0, 3/4, 1/4): {0: 45, 1: 42, 2: 34, 3: 42},
 (1/2, 0, 1/4, 3/4): {0: 45, 1: 42, 2: 34, 3: 42},
 (1/4, 3/4, 1/2, 0): {0: 34, 1: 42, 2: 45, 3: 42},
 (1/4, 3/4, 0, 1/2): {0: 34, 1: 42, 2: 45, 3: 42},
 (1/4, 1/2, 3/4, 0): {0: 37, 1: 39, 2: 37, 3: 29},
 (1/4, 1/2, 0, 3/4): {0: 37, 1: 39, 2: 37, 3: 29},
 (1/4, 0, 3/4, 1/2): {0: 27, 1: 34, 2: 37, 3: 34},
 (1/4, 0, 1/2, 3/4): {0: 27, 1: 34, 2: 37, 3: 34},
 (0, 3/4, 1/2, 1/4): {0: 37, 1: 29, 2: 37, 3: 39},
 (0, 3/4, 1/4, 1/2): {0: 37, 1: 29, 2: 37, 3: 39},
 (0, 1/2, 3/4, 1/4): {0: 45, 1: 42, 2: 34, 3: 42},
 (0, 1/2, 1/4, 3/4): {0: 45, 1: 42, 2: 34, 3: 42},
 (0, 1/4, 3/4, 1/2): {0: 27, 1: 34, 2: 37, 3: 34},
 (0, 1/4, 1/2, 3/4): {0: 27, 1: 34, 2: 37, 3: 34}}
\end{lstlisting}
Note that we see each block repeated 4 times; these correspond to the orbits under the order 4 stabilizer of the highest weight $\Lambda$ generated by $s_1$ and $s_3$.

Now, let us consider the example of \eqref{eq:sl3}:
\begin{lstlisting}
sage: e=3 
sage: multicharge=(0,0,1,2) 
sage: al={0:3,1:2,2:2}
sage: all_RoCKs(al,e,multicharge)
{(2/3, 1/3, 0): {0: 11, 1: 18, 2: 18},
 (2/3, 0, 1/3): {0: 23, 1: 22, 2: 14},
 (1/3, 2/3, 0): {0: 23, 1: 14, 2: 22},
 (1/3, 0, 2/3): {0: 23, 1: 22, 2: 14},
 (0, 2/3, 1/3): {0: 23, 1: 14, 2: 22},
 (0, 1/3, 2/3): {0: 11, 1: 18, 2: 18}}
\end{lstlisting}
Again, note that the blocks appear in pairs due to the symmetry with respect to $s_3$.  The chambers in the dominant Weyl chamber for $W_{\mu}$ are those with $h_3\geq h_1$, so the last three entries.  These can be identified with the red triangles of \eqref{eq:sl3} by considering which simple roots are negative on each chamber: $\al_2$ on the first chamber, $\al_1$ on the second, and both on the third.  

To find the actual underlying blocks, we apply to each block the function \verb|weight_from_block| which tells us the values of $t_i$ for each block.  Note that the last input simply shifts all the values simultaneously, which we have done to keep the charge of each runner positive.  
\begin{lstlisting}
sage: weight_from_block({0: 3, 1: 2, 2: 2},e,multicharge,11)
[[14, 12, 10], [13, 12, 11]]
sage: weight_from_block({0: 11, 1: 18, 2: 18},e,multicharge,11)
[[6, 12, 18], [13, 12, 11]]
sage: weight_from_block({0: 23, 1: 22, 2: 14},e,multicharge,11)
[[14, 20, 2], [13, 12, 11]]
sage: weight_from_block({0: 23, 1: 14, 2: 22},e,multicharge,11)
[[22, 4, 10], [13, 12, 11]]
\end{lstlisting}
Thus, one example of a multipartition in the dominant block is given by $((1,1),(2,1),(1,1),\emptyset)$, whose abacus we show below:

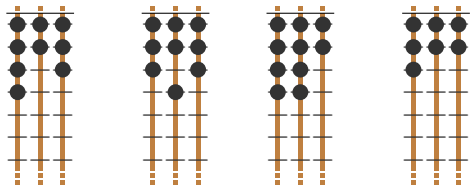
\begin{center}
\begin{tikzpicture}\tikzset{yscale=0.3,xscale=0.3}
\foreach \k in {0,1,2} {
\fill [color=brown] (\k-.1,0)--(\k-.1,7) -- (\k+0.1,7)-- (\k+.1,-0)--(\k-.1,-0);
\fill[color=brown](\k-.1,7.1)--(\k+.1,7.1)--(\k+.1,7.3)--(\k-.1,7.3)--(\k-.1,7.1);
\fill[color=brown](\k-.1,-0.1)--(\k+.1,-0.1)--(\k+.1,-0.3)--(\k-.1,-0.3)--(\k-.1,-0.3);
\fill[color=brown](\k-.1,-0.4)--(\k+.1,-0.4)--(\k+.1,-0.6)--(\k-.1,-0.6)--(\k-.1,-0.6);
}
\foreach \k in {3.5,4.5,5.5,6.5} {
\fill[color=bead] (0,\k) circle (10pt);}
\foreach \k in {4.5,5.5,6.5} {
\fill[color=bead] (2,\k) circle (10pt);}
\foreach \k in {5.5,6.5} {
\fill[color=bead] (1,\k) circle (10pt);}
\draw(-.5,7)--(2.5,7); 
\foreach \a in {0,1,2} {
\foreach \k in {.5,...,6.5}{
\node at (\a,\k)[color=bead]{$-$};};};
\end{tikzpicture}
\; \,
\begin{tikzpicture}\tikzset{yscale=0.3,xscale=0.3}
\foreach \k in {0,1,2} {
\fill [color=brown] (\k-.1,0)--(\k-.1,7) -- (\k+0.1,7)-- (\k+.1,0)--(\k-.1,0);
\fill[color=brown](\k-.1,7.1)--(\k+.1,7.1)--(\k+.1,7.3)--(\k-.1,7.3)--(\k-.1,7.1);
\fill[color=brown](\k-.1,-0.1)--(\k+.1,-0.1)--(\k+.1,-0.3)--(\k-.1,-0.3)--(\k-.1,-0.3);
\fill[color=brown](\k-.1,-0.4)--(\k+.1,-0.4)--(\k+.1,-0.6)--(\k-.1,-0.6)--(\k-.1,-0.6);
}
\foreach \k in {4.5,5.5,6.5} {
\fill[color=bead] (0,\k) circle (10pt);}
\foreach \k in {4.5,5.5,6.5} {
\fill[color=bead] (2,\k) circle (10pt);}
\foreach \k in {3.5,5.5,6.5} {
\fill[color=bead] (1,\k) circle (10pt);}
\draw(-.5,7)--(2.5,7); 
\foreach \a in {0,1,2} {
\foreach \k in {.5,...,6.5}{
\node at (\a,\k)[color=bead]{$-$};};};
\end{tikzpicture}\; \,
\begin{tikzpicture}\tikzset{yscale=0.3,xscale=0.3}
\foreach \k in {0,1,2} {
\fill [color=brown] (\k-.1,0)--(\k-.1,7) -- (\k+0.1,7)-- (\k+.1,0)--(\k-.1,0);
\fill[color=brown](\k-.1,7.1)--(\k+.1,7.1)--(\k+.1,7.3)--(\k-.1,7.3)--(\k-.1,7.1);
\fill[color=brown](\k-.1,-0.1)--(\k+.1,-0.1)--(\k+.1,-0.3)--(\k-.1,-0.3)--(\k-.1,-0.3);
\fill[color=brown](\k-.1,-0.4)--(\k+.1,-0.4)--(\k+.1,-0.6)--(\k-.1,-0.6)--(\k-.1,-0.6);
}
\foreach \k in {3.5,4.5,5.5,6.5} {
\fill[color=bead] (0,\k) circle (10pt);}
\foreach \k in {5.5,6.5} {
\fill[color=bead] (2,\k) circle (10pt);}
\foreach \k in {3.5,4.5, 5.5,6.5} {
\fill[color=bead] (1,\k) circle (10pt);}
\draw(-.5,7)--(2.5,7); 
\foreach \a in {0,1,2} {
\foreach \k in {.5,...,6.5}{
\node at (\a,\k)[color=bead]{$-$};};};
\end{tikzpicture}
\; \,
\begin{tikzpicture}\tikzset{yscale=0.3,xscale=0.3}
\foreach \k in {0,1,2} {
\fill [color=brown] (\k-.1,0)--(\k-.1,7) -- (\k+0.1,7)-- (\k+.1,0)--(\k-.1,0);
\fill[color=brown](\k-.1,7.1)--(\k+.1,7.1)--(\k+.1,7.3)--(\k-.1,7.3)--(\k-.1,7.1);
\fill[color=brown](\k-.1,-0.1)--(\k+.1,-0.1)--(\k+.1,-0.3)--(\k-.1,-0.3)--(\k-.1,-0.3);
\fill[color=brown](\k-.1,-0.4)--(\k+.1,-0.4)--(\k+.1,-0.6)--(\k-.1,-0.6)--(\k-.1,-0.6);
}
\foreach \k in {4.5,5.5,6.5} {
\fill[color=bead] (0,\k) circle (10pt);}
\foreach \k in {5.5,6.5} {
\fill[color=bead] (2,\k) circle (10pt);}
\foreach \k in {5.5,6.5} {
\fill[color=bead] (1,\k) circle (10pt);}
\draw(-.5,7)--(2.5,7); 
\foreach \a in {0,1,2} {
\foreach \k in {.5,...,6.5}{
\node at (\a,\k)[color=bead]{$-$};};};
\end{tikzpicture}
\end{center}
The corresponding objects in the different RoCK blocks above are given by 
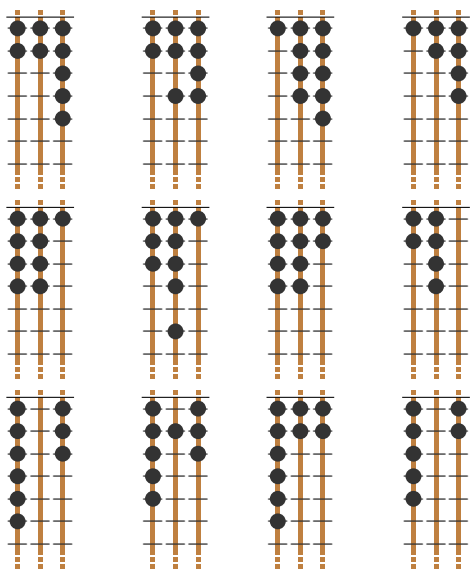
\definecolor{bead}{gray}{0.2}
\begin{center}
\begin{tikzpicture}\tikzset{yscale=0.3,xscale=0.3}
\foreach \k in {0,1,2} {
\fill [color=brown] (\k-.1,0)--(\k-.1,7) -- (\k+0.1,7)-- (\k+.1,0)--(\k-.1,0);
\fill[color=brown](\k-.1,7.1)--(\k+.1,7.1)--(\k+.1,7.3)--(\k-.1,7.3)--(\k-.1,7.1);
\fill[color=brown](\k-.1,-0.1)--(\k+.1,-0.1)--(\k+.1,-0.3)--(\k-.1,-0.3)--(\k-.1,-0.3);
\fill[color=brown](\k-.1,-0.4)--(\k+.1,-0.4)--(\k+.1,-0.6)--(\k-.1,-0.6)--(\k-.1,-0.6);
}
\foreach \k in {2.5,3.5,4.5,5.5,6.5} {
\fill[color=bead] (2,\k) circle (10pt);}
\foreach \k in {5.5,6.5} {
\fill[color=bead] (0,\k) circle (10pt);}
\foreach \k in {5.5,6.5} {
\fill[color=bead] (1,\k) circle (10pt);}
\draw(-.5,7)--(2.5,7); 
\foreach \a in {0,1,2} {
\foreach \k in {.5,...,6.5}{
\node at (\a,\k)[color=bead]{$-$};};};
\end{tikzpicture}
\; \,
\begin{tikzpicture}\tikzset{yscale=0.3,xscale=0.3}
\foreach \k in {0,1,2} {
\fill [color=brown] (\k-.1,0)--(\k-.1,7) -- (\k+0.1,7)-- (\k+.1,0)--(\k-.1,0);
\fill[color=brown](\k-.1,7.1)--(\k+.1,7.1)--(\k+.1,7.3)--(\k-.1,7.3)--(\k-.1,7.1);
\fill[color=brown](\k-.1,-0.1)--(\k+.1,-0.1)--(\k+.1,-0.3)--(\k-.1,-0.3)--(\k-.1,-0.3);
\fill[color=brown](\k-.1,-0.4)--(\k+.1,-0.4)--(\k+.1,-0.6)--(\k-.1,-0.6)--(\k-.1,-0.6);
}
\foreach \k in {5.5,6.5} {
\fill[color=bead] (0,\k) circle (10pt);}
\foreach \k in {3.5,4.5,5.5,6.5} {
\fill[color=bead] (2,\k) circle (10pt);}
\foreach \k in {3.5,5.5,6.5} {
\fill[color=bead] (1,\k) circle (10pt);}
\draw(-.5,7)--(2.5,7); 
\foreach \a in {0,1,2} {
\foreach \k in {.5,...,6.5}{
\node at (\a,\k)[color=bead]{$-$};};};
\end{tikzpicture}\; \,
\begin{tikzpicture}\tikzset{yscale=0.3,xscale=0.3}
\foreach \k in {0,1,2} {
\fill [color=brown] (\k-.1,0)--(\k-.1,7) -- (\k+0.1,7)-- (\k+.1,0)--(\k-.1,0);
\fill[color=brown](\k-.1,7.1)--(\k+.1,7.1)--(\k+.1,7.3)--(\k-.1,7.3)--(\k-.1,7.1);
\fill[color=brown](\k-.1,-0.1)--(\k+.1,-0.1)--(\k+.1,-0.3)--(\k-.1,-0.3)--(\k-.1,-0.3);
\fill[color=brown](\k-.1,-0.4)--(\k+.1,-0.4)--(\k+.1,-0.6)--(\k-.1,-0.6)--(\k-.1,-0.6);
}
\foreach \k in {5.5,6.5,2.5,3.5,4.5} {
\fill[color=bead] (2,\k) circle (10pt);}
\foreach \k in {3.5,4.5, 5.5,6.5} {
\fill[color=bead] (1,\k) circle (10pt);}
\foreach \k in {6.5} {
\fill[color=bead] (0,\k) circle (10pt);}
\draw(-.5,7)--(2.5,7); 
\foreach \a in {0,1,2} {
\foreach \k in {.5,...,6.5}{
\node at (\a,\k)[color=bead]{$-$};};};
\end{tikzpicture}
\; \,
\begin{tikzpicture}\tikzset{yscale=0.3,xscale=0.3}
\foreach \k in {0,1,2} {
\fill [color=brown] (\k-.1,0)--(\k-.1,7) -- (\k+0.1,7)-- (\k+.1,0)--(\k-.1,0);
\fill[color=brown](\k-.1,7.1)--(\k+.1,7.1)--(\k+.1,7.3)--(\k-.1,7.3)--(\k-.1,7.1);
\fill[color=brown](\k-.1,-0.1)--(\k+.1,-0.1)--(\k+.1,-0.3)--(\k-.1,-0.3)--(\k-.1,-0.3);
\fill[color=brown](\k-.1,-0.4)--(\k+.1,-0.4)--(\k+.1,-0.6)--(\k-.1,-0.6)--(\k-.1,-0.6);
}
\foreach \k in {3.5,4.5,5.5,6.5} {
\fill[color=bead] (2,\k) circle (10pt);}
\foreach \k in {5.5,6.5} {
\fill[color=bead] (1,\k) circle (10pt);}
\draw(-.5,7)--(2.5,7); 
\foreach \k in {6.5} {
\fill[color=bead] (0,\k) circle (10pt);}
\foreach \a in {0,1,2} {
\foreach \k in {.5,...,6.5}{
\node at (\a,\k)[color=bead]{$-$};};};
\end{tikzpicture}
\end{center}
\begin{center}
\begin{tikzpicture}\tikzset{yscale=0.3,xscale=0.3}
\foreach \k in {0,1,2} {
\fill [color=brown] (\k-.1,0)--(\k-.1,7) -- (\k+0.1,7)-- (\k+.1,0)--(\k-.1,0);
\fill[color=brown](\k-.1,7.1)--(\k+.1,7.1)--(\k+.1,7.3)--(\k-.1,7.3)--(\k-.1,7.1);
\fill[color=brown](\k-.1,-0.1)--(\k+.1,-0.1)--(\k+.1,-0.3)--(\k-.1,-0.3)--(\k-.1,-0.3);
\fill[color=brown](\k-.1,-0.4)--(\k+.1,-0.4)--(\k+.1,-0.6)--(\k-.1,-0.6)--(\k-.1,-0.6);
}
\foreach \k in {3.5,4.5,5.5,6.5} {
\fill[color=bead] (0,\k) circle (10pt);}
\foreach \k in {6.5} {
\fill[color=bead] (2,\k) circle (10pt);}
\foreach \k in {3.5,4.5,5.5,6.5} {
\fill[color=bead] (1,\k) circle (10pt);}
\draw(-.5,7)--(2.5,7); 
\foreach \a in {0,1,2} {
\foreach \k in {.5,...,6.5}{
\node at (\a,\k)[color=bead]{$-$};};};
\end{tikzpicture}
\; \,
\begin{tikzpicture}\tikzset{yscale=0.3,xscale=0.3}
\foreach \k in {0,1,2} {
\fill [color=brown] (\k-.1,0)--(\k-.1,7) -- (\k+0.1,7)-- (\k+.1,0)--(\k-.1,0);
\fill[color=brown](\k-.1,7.1)--(\k+.1,7.1)--(\k+.1,7.3)--(\k-.1,7.3)--(\k-.1,7.1);
\fill[color=brown](\k-.1,-0.1)--(\k+.1,-0.1)--(\k+.1,-0.3)--(\k-.1,-0.3)--(\k-.1,-0.3);
\fill[color=brown](\k-.1,-0.4)--(\k+.1,-0.4)--(\k+.1,-0.6)--(\k-.1,-0.6)--(\k-.1,-0.6);
}
\foreach \k in {5.5,6.5,4.5} {
\fill[color=bead] (0,\k) circle (10pt);}
\foreach \k in {6.5} {
\fill[color=bead] (2,\k) circle (10pt);}
\foreach \k in {1.5,3.5,4.5,5.5,6.5} {
\fill[color=bead] (1,\k) circle (10pt);}
\draw(-.5,7)--(2.5,7); 
\foreach \a in {0,1,2} {
\foreach \k in {.5,...,6.5}{
\node at (\a,\k)[color=bead]{$-$};};};
\end{tikzpicture}\; \,
\begin{tikzpicture}\tikzset{yscale=0.3,xscale=0.3}
\foreach \k in {0,1,2} {
\fill [color=brown] (\k-.1,0)--(\k-.1,7) -- (\k+0.1,7)-- (\k+.1,0)--(\k-.1,0);
\fill[color=brown](\k-.1,7.1)--(\k+.1,7.1)--(\k+.1,7.3)--(\k-.1,7.3)--(\k-.1,7.1);
\fill[color=brown](\k-.1,-0.1)--(\k+.1,-0.1)--(\k+.1,-0.3)--(\k-.1,-0.3)--(\k-.1,-0.3);
\fill[color=brown](\k-.1,-0.4)--(\k+.1,-0.4)--(\k+.1,-0.6)--(\k-.1,-0.6)--(\k-.1,-0.6);
}
\foreach \k in {5.5,6.5,3.5,4.5} {
\fill[color=bead] (0,\k) circle (10pt);}
\foreach \k in {5.5,6.5} {
\fill[color=bead] (2,\k) circle (10pt);}
\foreach \k in {3.5,4.5, 5.5,6.5} {
\fill[color=bead] (1,\k) circle (10pt);}
\draw(-.5,7)--(2.5,7); 
\foreach \a in {0,1,2} {
\foreach \k in {.5,...,6.5}{
\node at (\a,\k)[color=bead]{$-$};};};
\end{tikzpicture}
\; \,
\begin{tikzpicture}\tikzset{yscale=0.3,xscale=0.3}
\foreach \k in {0,1,2} {
\fill [color=brown] (\k-.1,0)--(\k-.1,7) -- (\k+0.1,7)-- (\k+.1,0)--(\k-.1,0);
\fill[color=brown](\k-.1,7.1)--(\k+.1,7.1)--(\k+.1,7.3)--(\k-.1,7.3)--(\k-.1,7.1);
\fill[color=brown](\k-.1,-0.1)--(\k+.1,-0.1)--(\k+.1,-0.3)--(\k-.1,-0.3)--(\k-.1,-0.3);
\fill[color=brown](\k-.1,-0.4)--(\k+.1,-0.4)--(\k+.1,-0.6)--(\k-.1,-0.6)--(\k-.1,-0.6);
}
\foreach \k in {5.5,6.5} {
\fill[color=bead] (0,\k) circle (10pt);}
\foreach \k in {} {
\fill[color=bead] (2,\k) circle (10pt);}
\foreach \k in {3.5,4.5,5.5,6.5} {
\fill[color=bead] (1,\k) circle (10pt);}
\draw(-.5,7)--(2.5,7); 
\foreach \a in {0,1,2} {
\foreach \k in {.5,...,6.5}{
\node at (\a,\k)[color=bead]{$-$};};};
\end{tikzpicture}
\end{center}
\begin{center}
\begin{tikzpicture}\tikzset{yscale=0.3,xscale=0.3}
\foreach \k in {0,1,2} {
\fill [color=brown] (\k-.1,0)--(\k-.1,7) -- (\k+0.1,7)-- (\k+.1,0)--(\k-.1,0);
\fill[color=brown](\k-.1,7.1)--(\k+.1,7.1)--(\k+.1,7.3)--(\k-.1,7.3)--(\k-.1,7.1);
\fill[color=brown](\k-.1,-0.1)--(\k+.1,-0.1)--(\k+.1,-0.3)--(\k-.1,-0.3)--(\k-.1,-0.3);
\fill[color=brown](\k-.1,-0.4)--(\k+.1,-0.4)--(\k+.1,-0.6)--(\k-.1,-0.6)--(\k-.1,-0.6);
}
\foreach \k in {1.5,2.5,3.5,4.5,5.5,6.5} {
\fill[color=bead] (0,\k) circle (10pt);}
\foreach \k in {4.5,5.5,6.5} {
\fill[color=bead] (2,\k) circle (10pt);}
\foreach \k in {} {
\fill[color=bead] (1,\k) circle (10pt);}
\draw(-.5,7)--(2.5,7); 
\foreach \a in {0,1,2} {
\foreach \k in {.5,...,6.5}{
\node at (\a,\k)[color=bead]{$-$};};};
\end{tikzpicture}
\; \,
\begin{tikzpicture}\tikzset{yscale=0.3,xscale=0.3}
\foreach \k in {0,1,2} {
\fill [color=brown] (\k-.1,0)--(\k-.1,7) -- (\k+0.1,7)-- (\k+.1,0)--(\k-.1,0);
\fill[color=brown](\k-.1,7.1)--(\k+.1,7.1)--(\k+.1,7.3)--(\k-.1,7.3)--(\k-.1,7.1);
\fill[color=brown](\k-.1,-0.1)--(\k+.1,-0.1)--(\k+.1,-0.3)--(\k-.1,-0.3)--(\k-.1,-0.3);
\fill[color=brown](\k-.1,-0.4)--(\k+.1,-0.4)--(\k+.1,-0.6)--(\k-.1,-0.6)--(\k-.1,-0.6);
}
\foreach \k in {2.5,3.5,4.5,5.5,6.5} {
\fill[color=bead] (0,\k) circle (10pt);}
\foreach \k in {4.5,5.5,6.5} {
\fill[color=bead] (2,\k) circle (10pt);}
\foreach \k in {5.5} {
\fill[color=bead] (1,\k) circle (10pt);}
\draw(-.5,7)--(2.5,7); 
\foreach \a in {0,1,2} {
\foreach \k in {.5,...,6.5}{
\node at (\a,\k)[color=bead]{$-$};};};
\end{tikzpicture}\; \,
\begin{tikzpicture}\tikzset{yscale=0.3,xscale=0.3}
\foreach \k in {0,1,2} {
\fill [color=brown] (\k-.1,0)--(\k-.1,7) -- (\k+0.1,7)-- (\k+.1,0)--(\k-.1,0);
\fill[color=brown](\k-.1,7.1)--(\k+.1,7.1)--(\k+.1,7.3)--(\k-.1,7.3)--(\k-.1,7.1);
\fill[color=brown](\k-.1,-0.1)--(\k+.1,-0.1)--(\k+.1,-0.3)--(\k-.1,-0.3)--(\k-.1,-0.3);
\fill[color=brown](\k-.1,-0.4)--(\k+.1,-0.4)--(\k+.1,-0.6)--(\k-.1,-0.6)--(\k-.1,-0.6);
}
\foreach \k in {1.5,2.5,3.5,4.5,5.5,6.5} {
\fill[color=bead] (0,\k) circle (10pt);}
\foreach \k in {5.5,6.5} {
\fill[color=bead] (2,\k) circle (10pt);}
\foreach \k in { 5.5,6.5} {
\fill[color=bead] (1,\k) circle (10pt);}
\draw(-.5,7)--(2.5,7); 
\foreach \a in {0,1,2} {
\foreach \k in {.5,...,6.5}{
\node at (\a,\k)[color=bead]{$-$};};};
\end{tikzpicture}
\; \,
\begin{tikzpicture}\tikzset{yscale=0.3,xscale=0.3}
\foreach \k in {0,1,2} {
\fill [color=brown] (\k-.1,0)--(\k-.1,7) -- (\k+0.1,7)-- (\k+.1,0)--(\k-.1,0);
\fill[color=brown](\k-.1,7.1)--(\k+.1,7.1)--(\k+.1,7.3)--(\k-.1,7.3)--(\k-.1,7.1);
\fill[color=brown](\k-.1,-0.1)--(\k+.1,-0.1)--(\k+.1,-0.3)--(\k-.1,-0.3)--(\k-.1,-0.3);
\fill[color=brown](\k-.1,-0.4)--(\k+.1,-0.4)--(\k+.1,-0.6)--(\k-.1,-0.6)--(\k-.1,-0.6);
}
\foreach \k in {2.5,3.5,4.5,5.5,6.5} {
\fill[color=bead] (0,\k) circle (10pt);}
\foreach \k in {5.5,6.5} {
\fill[color=bead] (2,\k) circle (10pt);}
\foreach \k in {} {
\fill[color=bead] (1,\k) circle (10pt);}
\draw(-.5,7)--(2.5,7); 
\foreach \a in {0,1,2} {
\foreach \k in {.5,...,6.5}{
\node at (\a,\k)[color=bead]{$-$};};};
\end{tikzpicture}
\end{center}
One can check by hand that the first of these blocks is Rouquier in the sense of \cite{lyleRouquier2022}, while the second and third become Rouquier after applying $\upsilon$ or $\upsilon^2$; note that this changes the highest weight.  Alternatively, we can find a Rouquier block by applying Lemma \ref{lem:Chuang-Rouquier} for transformations in the finite Weyl group to go to the blocks with $(t_1,t_2,t_3)=(2,14,20), (4,10,22)$, respectively.  
%\bibliography{RoCK}
%\bibliographystyle{amsalpha}
{\renewcommand{\markboth}[2]{}% Remove header adjustment
\printbibliography}
\end{document}